%% file: arxiv.tex
\def\arxivcombined{1}
\documentclass[10pt,letterpaper,leqno]{article}
\usepackage{arxivcompact}

\input{ex_shared}

\hypersetup{
  pdftitle={A Continuous-Time Actor-Critic Framework for Jump-Diffusion Controls with Normalizing Flows},
  pdfauthor={L. Guo, R. Hu, X. Yang, and Y. Zhu}
}

\begin{document}

\maketitle

\begin{abstract}
Continuous-time stochastic control with time-inhomogeneous jump--diffusion dynamics is central in finance and economics, but computing optimal policies is difficult under explicit time dependence, discontinuous shocks, and high dimensionality. We propose an actor--critic framework that serves as a mesh-free solver for entropy-regularized control problems and stochastic games with jumps. The approach is built on a time-inhomogeneous ``little'' $q$-function and an appropriate occupation measure, yielding a policy-gradient representation that accommodates time-dependent drift, volatility, and jump terms.
To represent expressive stochastic policies in continuous-action spaces, we parameterize the actor using conditional normalizing flows, enabling flexible non-Gaussian policies while retaining exact likelihood evaluation for entropy regularization and policy optimization. We validate the method on time-inhomogeneous linear--quadratic control, Merton portfolio optimization, and a multi-agent portfolio game, using explicit solutions or high-accuracy benchmarks. Numerical results demonstrate stable learning under jump discontinuities, accurate approximation of optimal stochastic policies, and favorable scaling with respect to dimension and number of agents.

\end{abstract}

\begin{keywords}
Jump diffusion, Stochastic Control, Reinforcement Learning, Hamilton--Jacobi--Bellman Equation, Actor-Critic Methods, Normalizing Flow
\end{keywords}

\begin{MSCcodes}
93E20, 60H10, 65C30, 91G10
\end{MSCcodes}

\section{Introduction}
Continuous-time stochastic control provides a mathematical framework for dynamical decision making in finance and economics \cite{pham2009continuous}. Many problems such as portfolio selection \cite{dai2023learning,merton1971optimum} can be formulated as controlling stochastic differential equations to maximize (or minimizing) an expected discounted objective. From a computational standpoint, however, classical approaches based on dynamic programming or stochastic maximum principles become difficult to implement when the state dimension is large \cite{CaiFangZhangZhou2024}, and when the underlying dynamics are unknown or only partially specified. These challenges have motivated the development of continuous-time reinforcement learning (RL) methods \cite{wang2020reinforcement,wang2020continuous,jia2022PE,jia2022policy} that combined with neural networks, aiming to learn near-optimal control directly from interaction with the environment without requiring explicit model structure and with improved scalability to higher dimensions. 

A growing literature has developed continuous-time analogs of policy evaluation and policy improvement. For policy evaluation, temporal-difference type schemes are derived in \cite{Doya2000,jia2022PE}, providing practical methods for approximating value function directly in continuous time. For policy improvement, \cite{jia2022policy} exploits martingale structure to rewrite policy-gradient objectives as policy-evaluation problems, yielding implementable update rules. From an action-value viewpoint, \cite{jia2023q} studies continuous-time $q$ learning and introduces a first-order surrogate, the ``little'' $q$-function, to avoid the degeneracy of the conventional ``big'' $Q$-function in the continuous-time limit \cite{tallec2019making}. Most of these developments are focused on finite-horizon criteria, with related extensions to mean-field control in \cite{wei2025continuous}. Infinite-horizon continuous-time policy-gradient formulas have appeared more recently, for example in \cite{zhao2023policy}.

In many financial settings, pure diffusion models are inadequate because asset prices and economic factors may exhibit abrupt movements driven by liquidity shocks or macroeconomic events \cite{oksendal2007applied,applebaum2009levy}. Incorporating jumps is therefore essential for capturing heavy tails, discontinuities, and jump risk premia for markets \cite{aydougan2025optimal,bender2023entropy}. This has motivated a growing body of work on deep learning and RL for jump-diffusion dynamics \cite{cheridito2025deep,bo2024continuous,gao2024reinforcement,lu2025multiagent,denkert2025control,guo2023reinforcement,jiang2024robust}. Motivated by the little-$q$ methodology, \cite{bo2024continuous,gao2024reinforcement} extend continuous-time q-learning ideas to stochastic policies and entropy regularization in jump-diffusion settings, while \cite{lu2025multiagent} develops an actor-critic method for deterministic controls in finite-horizon jump-diffusion games, and \cite{denkert2025control} considers optimal switching problems under jump dynamics. 

Most existing formulations, however, remain tied to finite-horizon objectives and often adopt Gaussian policy parameterizations in practice. In many situations, the optimal stochastic policy is non-Gaussian; see, for example, \cite{bo2024continuous}. This paper addresses these gaps by developing a learning framework for discounted infinite-horizon control of time-inhomogeneous jump-diffusions under general stochastic (possibly non-Gaussian) policies. Our first contribution is the introduction of a continuous-time little $q$-function and a time-dependent discounted occupation measure, and the establishment of structural properties that connect these objects to policy improvement. These results lead to a policy-gradient representation valid for general time-inhomogeneous jump-diffusions on an infinite horizon. For convenience, we provide a comparative summary of continuous-time $q$-function formulations in Table~\ref{tab:q-comparison}.

\begin{table}[htbp]
\centering
\small
\setlength{\tabcolsep}{3pt}
\caption{Comparison continuous-time $Q$/$q$-function formulations.}
\resizebox{\linewidth}{!}{$
\begin{tabular}{lccccc}
\toprule
Work & Time setting & State dynamics & Learning objects & Role of $q$ function & Policy class \\
\midrule

\cite{jia2023q}
& \multirow{2}{*}[\MRshift]{\makecell[c]{Continuous,\\ finite horizon}}
& Diffusion
& \makecell{$V_\psi(t, \bm x)$,\\ $q_\theta(t,\bm x,\bm u)$}
& \multirow{2}{*}[\MRshift]{\makecell[c]{%
$\pi^*(\bm u\mid t,\bm x)\propto
\exp\left\{\frac{1}{\gamma} q^*(t,\bm x,\bm u)\right\}$\\
$q$ Learned via \\ martingale orthogonality}}
& \multirow{2}{*}[\MRshift]{\makecell[c]{Gaussian}}
\\ 
\addlinespace[0.3em]
\cmidrule(lr){1-1}\cmidrule(lr){3-4}

\cite{bo2024continuous,gao2024reinforcement}
&
& Jump-diffusion
& \makecell{$V_\psi(t,\bm x)$,\\ $q_\theta(t,\bm x,\bm u)$}
&
\\
\addlinespace[0.3em] 
\midrule

\cite{zhao2023policy}
& \makecell{Continuous,\\ time-homog.,\\ infinite horizon}
& Diffusion
& \makecell{$V_\psi(\bm x)$,\\ $\pi_\theta(\bm x)$}
& \makecell[l]{%
Policy gradient thm for $\pi^\ast$ \\
$q$ approximated by GAE} & Gaussian 
\\
\midrule 

This work
& \makecell{Continuous,\\ time-inhomog.,\\ infinite horizon}
& Jump-diffusion
& \makecell{$V_\psi(t,\bm x)$,\\ $\pi_\theta(t,\bm x)$}
& \makecell[l]{%
Policy gradient thm for $\pi^\ast$\\
$q$ approximated by GAE} 
& \makecell[c]{General \\ (normalizing flow)}
\\

\bottomrule
\end{tabular}
$}
\label{tab:q-comparison}
\end{table}

Our second contribution is to provide tractable benchmarks by deriving explicit solutions in several canonical specifications, including linear-quadratic control, the Merton problem with jumps, and multi-agent games with jump-driven CARA utilities, together with representative time-inhomogeneous variants. These closed-form policies serve as ground truth for assessing numerical accuracy. Finally, we propose an implementable actor--critic algorithm that combines the derived policy-gradient representation with a conditional normalizing-flow parameterization of stochastic policies. The flow-based construction enables expressive, non-Gaussian distributions for controls while preserving tractable likelihoods and gradients, which is essential in the presence of entropy regularization and policy-gradient-type theorems. Numerical experiments in both low- and high-dimensional regimes demonstrate stable learning behavior across a range of time-dependent jump-diffusion models.

The rest of the paper is organized as follows. Section~\ref{sec:setup} introduces the problem setting, including the classical jump-diffusion stochastic control problem and its entropy-regularized formulation. Section~\ref{sec:method} presents the proposed actor-critic framework: Section~\ref{sec:critic} develops policy evaluation for the critic, Section~\ref{sec:actor} introduces the ``little'' $q$-function and the occupation measure, and develops policy improvement for the actor with its theoretical justification, and Section~\ref{sec:numerical_implementation} details the conditional normalizing flow  parameterization for the actor.
Section~\ref{sec:analytical_examples} states explicit solutions for several canonical problems and reports numerical experiments. We conclude in Section~\ref{sec:conclusion}. 

\section{Problem Setup}
\label{sec:setup}

Let $(\Omega,\mathcal F, \mathbb{F}:=(\mathcal F_t)_{t\ge 0},\mathbb P)$ be a filtered probability space satisfying the usual conditions.
Let $\bm W=(\bm W_t)_{t\ge 0}$ be a $d$-dimensional Brownian motion, $N(\mathrm{d}t, \mathrm{d} \bm z)$ be a Poisson random measure corresponding to a L\'evy process $(\bm L_t)_{t\geq 0}$, and $\nu$ be the L\'evy measure on $\mathbb R^d$ satisfying the integrability condition $\int_{\mathbb{R}^d} \min\{|\bm z|^2, 1\}\nu(\mathrm{d}\bm z)<\infty$ \cite{Duan2015Stochastic}.
The  associated compensated Poisson random measure is defined as  $\tilde N(\mathrm{d}t, \mathrm{d} \bm z):=N(\mathrm{d}t, \mathrm{d} \bm z)-\nu(\mathrm{d} \bm z)\,\mathrm{d}t$ and we assume that $\bm W$ and $N$ are independent.

We first recall the classical strict-control formulation. Let
$\bm u=(\bm u_t)_{t\geq0}$ be an admissible progressively measurable
control process taking values in
$\mathcal A\subset\mathbb R^m$. The controlled state process
$\bm X=(\bm X_t)_{t\geq0}\in\mathbb R^d$ follows the It\^o--L\'evy SDE
\begin{equation}\label{eq:controlled-SDE-levy}
     \mathrm{d}\bm X_{t} = \bm b(t,\bm X_{t-},\bm u_t)\,\mathrm{d}t
     + \bm\sigma(t,\bm X_{t-},\bm u_t)\,\mathrm{d}\bm W_t + \int_{\mathbb{R}^d} \bm\alpha(t,\bm X_{t-},\bm u_t,\bm z)\,\tilde N(\mathrm{d}t, \mathrm{d} \bm z)\,.
\end{equation}
For a given admissible control process, this SDE is well posed under the usual Lipschitz, linear-growth, and jump-integrability conditions \cite{oksendal2007applied}.

In reinforcement learning, it is natural to consider randomized (stochastic) policies, under which the action at $(t,\bm x)$ is drawn from a state-dependent distribution $\pi(\cdot\mid t,\bm x)\in\mathcal P(\mathcal A)$. However, as discussed in \cite{jia2025accuracy,szpruch2024optimal}, a measurability issue arises: to make the above SDE well posed under randomized policy, one needs a process $\bm u$ that is $\mathcal F$-progressively measurable and satisfies $\bm u_t\mid(\bm X_t=\bm x)\sim\pi(\cdot\mid t,\bm x)$ for each $t$. Such a construction is not immediate on a fixed stochastic basis, since time is uncountable and one cannot literally ``sample independently at every instant''. To avoid this issue, following \cite{jia2025accuracy}, we work with the exploratory state process
\begin{equation}
\label{eq:explor_SDE_PO}
\resizebox{\linewidth}{!}{$
\begin{aligned}
\mathrm{d}{\bm X}_t^\pi= \tilde{\bm b}(t,{\bm X}_{t-}^\pi;\pi)\,\mathrm{d}t
 + \tilde{\bm\sigma}(t,{\bm X}_{t-}^\pi;\pi)\,\mathrm{d}\bm W_t + \int_{\mathbb R^d\times[0,1]^m}
\bm\alpha\!\bigl(t,{\bm X}_{t-}^\pi,\; G^\pi(t,{\bm X}_{t-}^\pi,\bm r),\; \bm z\bigr)\,
\tilde N(\mathrm{d}t,\mathrm{d} \bm z,\mathrm{d} \bm r),
\end{aligned}
$}
\end{equation}
where $G^\pi:[0,\infty)\times\mathbb R^d\times[0,1]^m\to\mathcal A$ is a measurable function 
such that $(G^\pi(t,\bm x,\cdot))_{\#}\mathcal U \\=\pi(\cdot\mid t,\bm x)$, when $\mathcal U$ is the Lebesgue probability measure on $[0,1]^m$, $N(\mathrm{d}t,\mathrm{d} \bm z,\mathrm{d} \bm r)$ is a Poisson random measure on $(0,\infty)\times\mathbb R^d\times[0,1]^m$ with compensator $\nu(\mathrm{d} \bm z)\,\mathcal U(\mathrm{d} \bm r)\,\mathrm{d}t$, independent of the Brownian motion $\bm W$, $\tilde N(\mathrm{d}t,\mathrm{d} \bm z,\mathrm{d} \bm r)\!
:=\! N(\mathrm{d}t,\mathrm{d} \bm z,\mathrm{d} \bm r)-\nu(\mathrm{d} \bm z)\,\mathcal U(\mathrm{d} \bm r)\,\mathrm{d}t$ is the compensated measure, and $\tilde {\bm b}, \tilde {\bm \Sigma}$ are defined as
\begin{equation}
\label{eq:tilde_b_tilde_sigma}
\resizebox{\linewidth}{!}{$
\begin{aligned}
\tilde{\bm b}(t,\bm x;\pi):= \int_{\mathcal A} \bm b(t,\bm x,\bm u)\,\pi(\bm u\mid t,\bm x) \mathrm{d} \bm u,\,\,
\tilde{\bm{\Sigma}} = \tilde{\bm\sigma}^2(t,\bm x;\pi):= \int_{\mathcal A} \bm\sigma(t,\bm x,\bm u)\bm\sigma(t,\bm x,\bm u)^\top\,\pi(\bm u\mid t,\bm x) \mathrm{d} \bm u\,,
\end{aligned}
$}
\end{equation}
For each admissible policy $\pi$, we assume that $\bm{\tilde b}(\cdot,\cdot;\pi)$, $\bm{\tilde\sigma}(\cdot,\cdot;\pi)$, and $(t,\bm x,\bm z,\bm r)\mapsto \bm\alpha(t,\bm x,G^\pi(t,\bm x,\bm r),\bm z)$ satisfy the standard Lipschitz and linear-growth conditions, with those for the jump coefficient understood in $L^2(\nu\otimes\mathcal U)$. Then~\eqref{eq:explor_SDE_PO} admits a unique strong solution; see \cite[Proposition~1]{gao2024reinforcement}. Note that the exploratory SDE is an averaged continuous-time representation of the randomized policy and is not, in general, the directly observable state trajectory generated by a single realization of sampled actions. In the numerical implementation, actions are instead sampled on a discrete-time grid and held fixed between consecutive grid points. The resulting grid-sampled dynamics is described in Section~\suppref{sec:discretization} (supplementary materials). Under the assumptions of \cite{bender2026continuous}, it converges weakly, as the grid size tends to zero, to a limiting process corresponding in law to the exploratory SDE~\eqref{eq:explor_SDE_PO}.

With the exploratory state dynamics \eqref{eq:explor_SDE_PO}, we consider the following entropy-regularized reward:
\begin{equation}
\label{eq:tilde-f-def}
\tilde f(s,\bm y;\pi):= \int_{\mathcal A}
\big(
f(s,\bm y,\bm u)-\gamma\log\pi(\bm u\mid s,\bm y)
\big)\,\pi(\bm u\mid s,\bm y)\mathrm{d} \bm u, 
\end{equation}
where $f$, continuous at $(s, \bm y)$ and locally uniformly integrable along $\bm y$, is the standard running reward. We consider $\mathcal S\big(\pi(\cdot\mid s,\bm y)\big) := -\int_{\mathcal A}\pi(\bm u\mid s,\bm y)\log\pi(\bm u\mid s,\bm y)\,\mathrm{d} \bm u$,
the Shannon entropy, which encourages randomized policies and may improve numerical stability. For long-term control, let $\beta>0$ be a discount factor. We then define the entropy-regularized expected discounted reward by
\begin{equation}
\label{eq:Jgamma-shift}
\begin{aligned}
J(t,\bm x;\pi) &:= \mathbb{E}\!\Big[
\int_{t}^{\infty} e^{-\beta (s-t)}
\tilde f\big(s, \bm X_s^\pi; \pi\big)\,
\mathrm{d}s \mid \bm X_{t}^\pi = \bm x
\Big],
\end{aligned}
\end{equation}
where $\bm X_s^\pi$ solves the exploratory dynamics \eqref{eq:explor_SDE_PO}, and $\gamma \geq 0$ characterizes the intensity of regularization. 

For a fixed policy $\pi$, the function $J(t,\bm x;\pi)$ satisfies (cf. \cite[Lemma~3]{jia2022policy})
\begin{equation}
\label{eq:policy-eval-PIDE}
0
=\partial_t J(t,\bm x;\pi)
+ \mathcal L^{\pi} J(t,\bm x;\pi)
+ \tilde f(t,\bm x;\pi)
-\beta J(t,\bm x;\pi),
\end{equation}
where $\mathcal L^{\pi}$ is the $\pi$-averaged infinitesimal generator 
\begin{equation}
\label{eq:gen-pi}
(\mathcal L^\pi \varphi)(t,\bm x)
:= \int_{\mathcal A} (\mathcal L^{\bm u}\varphi)(t,\bm x)\,\pi(\bm u\mid t,\bm x) \mathrm{d} \bm u,
\quad \varphi\in C_c^{1,2}([0,T]\times\mathbb R^d),\ \forall\,T>0\,,
\end{equation}
and $\mathcal L^{\bm u}$ is the generator for fixed control sampling
\begin{align} 
(\mathcal L^{\bm u}\varphi)(t,\bm x)
& = \bm b(t,\bm x,\bm u)\cdot\nabla_{\bm x}\varphi(t,\bm x)
   + \tfrac12\operatorname{Tr}\big(\bm\sigma(t,\bm x,\bm u)\bm\sigma(t,\bm x,\bm u)^\top\nabla_{\bm x}^2\varphi(t,\bm x)\big)\nonumber \\ \label{eq:generator-general-levy}
&\,
   + \int_{\mathbb R^{d}}
      \big(
        \varphi\big(t,\bm x+\bm\alpha(t,\bm x,\bm u,\bm z)\big) - \varphi(t,\bm x)
        - \bm\alpha(t,\bm x,\bm u,\bm z)\cdot\nabla_{\bm x}\varphi(t,\bm x) \big)\,\nu(\mathrm{d} \bm z).
\end{align}
The optimal value function is
\begin{equation}
\label{eq:value_func_soft}
V(t,\bm x) :=\sup_{\pi} J(t,\bm x;\pi) =\sup_{\pi}\,
\mathbb E\!\Big[
\int_t^\infty e^{-\beta(s-t)}\,
\tilde f\big(s, \bm X_s^\pi;\pi\big)\,\mathrm{d}s \mid \bm X_t^\pi=\bm x
\Big],
\end{equation}
which satisfies the entropy-regularized Hamilton-Jacobi-Bellman (HJB) equation by dynamic programming (cf. \cite{oksendal2007applied})
\begin{equation}
\label{eq:soft-HJB-opt-general}
\begin{aligned}
0 = \partial_t V(t,\bm x)
&+ \sup_{\pi(\cdot\mid t,\bm x)}
\int_{\mathcal A}
\Big[
\mathscr{H}\big(
t,\bm x,\bm u, V(t, \bm x), \nabla_{\bm x} V(t,\bm x),\nabla_{\bm x}^2 V(t,\bm x)
\big)\\
& \qquad -\gamma\log\pi(\bm u\mid t,\bm x)
\Big]\, \pi(\bm u\mid t,\bm x)\mathrm{d} \bm u
-\beta V(t,\bm x).
\end{aligned}
\end{equation}
Here 
$\mathscr H:[0,\infty)\times\mathbb R^{d}\times\mathcal A\times \mathbb R \times \mathbb R^{d}\times\mathbb S^{d}\to\mathbb R$ is the Hamiltonian defined by
\begin{align}
\mathscr{H}(t,\bm x,\bm u, &V(t, \bm x),\nabla_{\bm x} V(t,\bm x), \nabla_{\bm x}^2 V(t,\bm x))
:= f(t,\bm x,\bm u) + \bm b(t,\bm x,\bm u)\cdot \nabla_{\bm x} V(t,\bm x)\nonumber \\ \label{eq:Hamiltonian-general}
& + \tfrac12 \operatorname{Tr}\big(\bm\sigma(t,\bm x,\bm u)\bm\sigma(t,\bm x,\bm u)^\top \nabla_{\bm x}^2 V(t,\bm x)\big) \\
& + \int_{\mathbb R^{d}}
\big(
V\big(t,\bm x+\bm\alpha(t,\bm x,\bm u,\bm z)\big) - V(t,\bm x)
- \bm\alpha(t,\bm x,\bm u,\bm z)\cdot \nabla_{\bm x} V(t,\bm x)
\big)\,\nu(\mathrm{d} \bm z)\,, \nonumber
\end{align}
and $\mathbb S^d$ denotes the space of real $d\times d$ symmetric matrices. Therefore, the optimal policy $\pi^*$ is solved as a maximizer of the $\sup$ part in \eqref{eq:soft-HJB-opt-general}.

Note that when $\gamma = 0$ in \eqref{eq:tilde-f-def}, the stochastic control problem reduces to the standard problem with deterministic control. In this case, the state process is governed by the classical controlled It\^o--L\'evy SDE with an admissible progressively measurable control process $\bm u=(\bm u_s)_{s\ge t}$, taking values in $\mathcal A\subset\mathbb R^m$, and the associated HJB equation becomes
\begin{equation}
\label{eq:HJB-hard}
    0
    = \partial_t V(t,\bm x)
      + \sup_{\bm u\in\mathcal A}
        \big\{
            (\mathcal L^{\bm u}V)(t,\bm x)
            + f(t,\bm x,\bm u)
        \big\}
      - \beta V(t,\bm x)\,,
\end{equation}
where the generator $\mathcal L^{\bm u}$ is defined as \eqref{eq:generator-general-levy}. 

\section{Actor-Critic for Time Inhomogeneous Jump Diffusion Control}
\label{sec:method}
We solve the infinite-horizon stochastic control problem via reinforcement learning (RL) using an actor-critic framework 
\cite{barto2012neuronlike}.
In RL, the \emph{actor} usually refers to the randomized policy $\pi$ and the \emph{critic} refers to the value $J(\cdot;\pi)$, used to evaluate the goodness of the current policy $\pi$. The actor-critic method consists of two steps: policy evaluations for the critic and policy improvement for the actor. By performing them interactively, one hopes to reach the optimal policy and value function $V$. 

Most existing continuous-time RL work is developed for finite horizons and/or dynamics driven only by Brownian motions
\cite{jia2022policy,jia2023q,jia2022PE,wang2020reinforcement}.
In the infinite-horizon setting considered here, the policy update step cannot be reduced to maximizing a finite-interval objective
in the same manner as in finite-horizon formulations \cite{lu2025multiagent}.
This necessitates a policy improvement principle that is consistent with discounting and time inhomogeneity, and remains valid under jump-diffusion dynamics.
The resulting actor update scheme is developed in Section~\ref{sec:actor}, while the critic is introduced below.

\subsection{Critic: Policy Evaluation}
\label{sec:critic}
Consider the critic $V_\psi$, parameterized by $\psi$. 
Our goal is to learn an accurate value approximation from incremental samples, without explicitly solving the partial integro-differential equation (PIDE) \eqref{eq:policy-eval-PIDE}. To this end, we use the continuous-time Bellman principle \cite{jia2022PE}, which leads to temporal-difference learning: we update the critic by minimizing TD errors computed along sampled trajectories.

\smallskip
\noindent\textbf{Bellman equation and TD error.}
Fix a stochastic feedback policy $\pi$.
Recall the entropy-regularized performance functional $J$ in \eqref{eq:Jgamma-shift} and the entropy-regularized reward $\tilde f(\cdot;\pi)$ in \eqref{eq:tilde-f-def}.
For any fixed deterministic $\delta_t>0$, by the law of total expectation and the Markov property under $\pi$, the discounted performance $J$ satisfies the Bellman equation
\begin{equation}
\label{eq:ct_bellman_policy_J_corrected}
\begin{aligned}
J(t,\bm X_t^\pi;\pi)
= \mathbb{E}\!\Big[
&\int_t^{t+\delta_t} e^{-\beta(s-t)}
\int_{\mathcal A}
\big(
f(s,\bm X_{s-}^\pi,\bm u)
-\gamma\log\pi(\bm u\mid s,\bm X_{s-}^\pi)
\big)\,\\
&\quad \pi(\bm u\mid s,\bm X_{s-}^\pi) \mathrm d\bm u\,\mathrm ds +e^{-\beta\delta_t}\,J(t+\delta_t,\bm X_{t+\delta_t}^\pi;\pi)
\ \big|\ \mathcal F_t\Big].
\end{aligned}
\end{equation}
We then define the one-step TD error over $[t,t+\delta_t)$ by
\begin{align}
\delta_{\mathrm{TD}}^{t}
:=
&\int_t^{t+\delta_t} e^{-\beta(s-t)}
\int_{\mathcal A}
\big(
f(s,\bm X_{s-}^\pi,\bm u)
-\gamma\log\pi(\bm u\mid s,\bm X_{s-}^\pi)
\big)\,
\pi(\bm u\mid s,\bm X_{s-}^\pi) \mathrm d\bm u\,\mathrm ds \nonumber \\ \label{eq:ct_TD_error_J}
&\quad + e^{-\beta\delta_t}V_\psi(t+\delta_t,\bm X_{t+\delta_t}^\pi;\pi)
- V_\psi(t,\bm X_t^\pi;\pi),
\end{align}
and \eqref{eq:ct_bellman_policy_J_corrected} implies that, for the critic $V_\psi$ that evaluates $J$ exactly, $\mathbb E \big[\delta_{\mathrm{TD}}^{t}\mid\mathcal F_t\big]=0.$

\smallskip
\noindent\textbf{Martingale-corrected TD error.}
By It\^o's formula, the last two terms in the one-step TD error admit the decomposition
\[
\begin{aligned}
&e^{-\beta\delta_t}V_\psi(t+\delta_t,\bm X_{t+\delta_t}^\pi;\pi) - V_\psi(t,\bm X_t^\pi;\pi)\\
&\qquad \qquad=
\int_t^{t+\delta_t} e^{-\beta(s-t)}
\big(\partial_s V_\psi + \mathcal L^{\pi}V_\psi - \beta V_\psi\big)(s,\bm X_{s-}^\pi;\pi)\,\mathrm{d}s
\;+\;\mathcal{I}_{t,t+\delta_t}^{\pi}.
\end{aligned}
\]
where $\mathcal{I}_{t,t+\delta_t}^{\pi}$ defined as
\begin{equation}
\label{eq:mart_inc_J_corrected}
\resizebox{\linewidth}{!}{$
\begin{aligned}
\mathcal{I}_{t,t+\delta_t}^{\pi}
:=\;& \int_t^{t+\delta_t} e^{-\beta(s-t)}
      \big(\bm\sigma(s,\bm X_{s-}^\pi,\bm u_s)^\top \nabla_x V_\psi(s,\bm X_{s-}^\pi;\pi)\big)^\top \mathrm{d}\bm W_s \\
&+\int_t^{t+\delta_t}\!\!\int_{\mathbb{R}^d} e^{-\beta(s-t)} \big[
    V_\psi\big(s,\bm X_{s-}^\pi+\bm\alpha(s,\bm X_{s-}^\pi,\bm u_s,\bm z);\pi\big)
    - V_\psi(s,\bm X_{s-}^\pi;\pi)
\big]\tilde N(\mathrm{d}s,\mathrm{d} \bm z),
\end{aligned}
$}
\end{equation}
captures the instantaneous fluctuations induced by Brownian and jump noises. As discussed in \cite{zhou2021actor,lu2025multiagent, zhou2025learning}, it has mean zero but adds extra variance to the learning signal. Therefore, subtracting $\mathcal{I}_{t,t+\delta_t}^{\pi}$ from the one-step TD error reduces variance while preserving unbiasedness.

Accordingly, we define the martingale-corrected TD error by:
\begin{equation}
\label{eq:ct_martingale_corrected_TD}
\resizebox{\linewidth}{!}{$
\tilde\delta_{\mathrm{TD}}^{t}
:=\delta_{\mathrm{TD}}^{t} -\mathcal{I}_{t,t+\delta_t}^{\pi}
=
\int_t^{t+\delta_t} e^{-\beta(s-t)}
    \big(
      \tilde f(s,\bm X_{s-}^\pi;\pi)
        + \partial_s V_\psi + \mathcal L^{\pi} V_\psi - \beta V_\psi
    \big)(s,\bm X_{s-}^\pi;\pi)\,\mathrm{d}s.
    $}
\end{equation}
If the critic $V_\psi$ evaluates $J$ exactly, i.e., it  solves the PIDE \eqref{eq:policy-eval-PIDE}, then $\tilde\delta_{\mathrm{TD}}^{t} = 0$ a.s.. 

\subsection{Actor: Policy Improvement}
\label{sec:actor}
Policy improvement updates the actor using the critic's value information to increase the expected discounted reward. A popular approach is policy gradient: the actor is parameterized (for example by neural networks), and the critic is used to construct an estimator of the gradient of the objective with respect to the actor parameters, which then drives the actor update.

In discrete-time with state/action space, the action-value $Q$ function is a common choice for this purpose. In continuous time and state/action space, however, a direct analogue of discrete-time $Q$-learning is intrinsically delicate:
the standard (``capital'' $Q$) action-value function degenerates to the value function, and na\"ive discretization-based updates can be highly sensitive to the time step~\cite{schulman2015high,jia2023q}. These issues motivate the ``little'' $q$-formulation advocated in~\cite{jia2023q,zhao2023policy,bo2024continuous}.

Following this line, we introduce a time-inhomogeneous ``little'' $q$-function for infinite-horizon jump-diffusion control (Section~\ref{sec:q_func}) and derive the corresponding policy gradient theorem  (Theorem~\ref{thm:PG} in Section~\ref{sec:PG}).
Because the resulting gradient depends on $q(\cdot;\pi)$ and is not directly implementable from data, we adopt a generalized advantage estimator justified by Lemma~\ref{lem:gae-approximation}. This extension for time-inhomogeneous is nontrivial: unlike the critic, which extends by standard time augmentation, the actor update is not a mere ``add~$t$'' modification, since explicit time dependence interacts with discounting and reshapes the occupation measure on $[t,\infty)\times\mathbb R^d$; this is why the $\frac{\partial J}{\partial t}$ term must enter $q$.

\subsubsection{Occupation Measure and $q$-Function}
\label{sec:q_func}

To accommodate possible time inhomogeneity in the  infinite-horizon setting, we first define a discounted occupation measure on $[t, \infty) \times \mathbb{R}^d$, extending \cite[Def.~2]{zhao2023policy}.
This definition is the $\beta$-potential of $\bm X^\pi$ and characterizes the discounted visitation frequencies of the time-state process starting from $(t,\bm x)$.

\begin{definition}
\label{def:beta_discount-inhom-shift}
Fix $\beta>0$. Let $\pi$ be an admissible stochastic policy such that, for every
$(t,\bm x)\in[0,\infty)\times\mathbb R^d$, the exploratory dynamics
\eqref{eq:explor_SDE_PO} starting from $\bm X_t^\pi=\bm x$ admits a unique strong solution $(\bm X_s^\pi)_{s\ge t}$.
The $\beta$-discounted occupation measure of $\bm X^\pi$ is defined by
\begin{equation}
\label{eq:mu-time-inhom-shift-def}
\mu^{\pi,t,\bm x}(A)
:= \mathbb E\Big[\int_{t}^{\infty} e^{-\beta (s-t)}
\,\mathbf 1_{\{(s,\bm X_s^\pi)\in A\}}\,\mathrm{d}s \Big],
\qquad
A\in \mathcal B([t,\infty)\times\mathbb R^d).
\end{equation}
This is a finite measure on $[t, \infty) \times \mathbb R^d$ with total mass $\mu^{\pi,t,\bm x}\big([t,\infty)\times\mathbb R^d\big) =
\mathbb E[\int_{t}^{\infty}$ $e^{-\beta (s-t)}\,\mathrm{d}s ]=\int_{t}^{\infty} e^{-\beta (s-t)}\,\mathrm{d}s = \beta^{-1}.$ Unless otherwise stated, expectations are taken under the path measure induced by the policy currently under discussion.

\end{definition}
 
We next derive the little $q$-function $q(t,\bm x,\bm u;\pi)$, which quantifies the instantaneous advantage of taking action $\bm u$ at $(t, \bm x)$ and then reverting to the current policy $\pi$. We sketch the main idea and defer the details to supplementary materials Section~\suppref{appen:derivation_q}.

Fix $\delta_t>0$, $\bm u\in\mathcal A$, and a baseline policy $\pi$. We consider a perturbed control: on the short interval $[t,t+\delta_t)$, we apply the constant action $\bm u$, and for $s \geq t+\delta_t$, we follow $\pi$. 
Let $(\bm X_s^{\bm u})_{s\ge t}$ denote the resulting state process with $\bm X_t^{\bm u}=\bm x$ (i.e., it solves the strict-control SDE \eqref{eq:controlled-SDE-levy} on $[t,t+\delta_t)$ with action $\bm u$ and then the exploratory SDE \eqref{eq:explor_SDE_PO} under $\pi$ on $[t+\delta_t,\infty)$, initialized at $(t+\delta_t,\bm X_{t+\delta_t}^{\bm u})$).
Define the corresponding discounted reward $Q_{\delta_t}(t,\bm x,\bm u;\pi)$ by
\begin{equation}
\resizebox{\linewidth}{!}{$
\begin{aligned}
\label{eq:Q_delta}
Q_{\delta_t}(t, \bm x, \bm u ; \pi)
:= \mathbb{E}\Big[\int_t^{t+\delta_t} e^{-\beta(s-t)}\, f\bigl(s, \bm X_s^{\bm u}, \bm u \bigr)\,\mathrm{d}s +\int_{t+\delta_t}^\infty e^{-\beta(s-t)}
\tilde f\bigl(s, \bm X_s^{\bm u}; \pi \bigr)\,\mathrm{d}s \mid \bm X_t^{\bm u} = \bm x
\Big],
\end{aligned}
$}
\end{equation}
A first-order expansion (see Section~\suppref{appen:derivation_q} in the supplementary materials) yields
\begin{align}\label{eq:Qderive}
&Q_{\delta_t}\bigl(t, \bm x, \bm u ; \pi\bigr) = J\bigl(t, \bm x; \pi\bigr) + \Bigl(
\partial_t J\bigl(t, \bm x; \pi\bigr) \\ 
& + \mathscr{H}\bigl(t, \bm x,\bm u, J(t,\bm x; \pi), \nabla_x J\bigl(t, \bm x ; \pi\bigr), \nabla_x^2 J\bigl(t, \bm x; \pi\bigr) \bigr)
- \beta J\bigl(t, \bm x ; \pi\bigr)
\Bigr)\,\delta_t + o(\delta_t),\nonumber 
\end{align}
where $\mathscr{H}$ is the Hamiltonian defined in \eqref{eq:Hamiltonian-general}. This motivates the definition of the (little) $q$-function:
\begin{equation}
\label{eq:q_function}
\resizebox{\linewidth}{!}{$
    q(t,\bm x,\bm u;\pi)
    := \partial_t J(t,\bm x; \pi)
    + \mathscr{H}\bigl(t, \bm x,\bm u, J(t,\bm x; \pi), \nabla_x J(t,\bm x; \pi), \nabla_x^2 J(t,\bm x; \pi)\bigr)
    - \beta J(t,\bm x; \pi).
    $}
\end{equation}
Indeed, $q(t,\bm x,\bm u;\pi)$ is the leading-order marginal gain per unit time of deviating from $\pi$ to $\bm u$.

Compared with \cite{zhao2023policy}, our definition includes the additional time-derivative term $\frac{\partial J}{\partial t}$ arising from time inhomogeneity. Compared with \cite{jia2023q}, we incorporate jump-diffusion dynamics and an infinite-horizon discounted objective through the Hamiltonian. 
Closely related jump extensions of the little-$q$ framework include \cite{bo2024continuous,gao2024reinforcement}; \cite{bo2024continuous} focuses on Poisson point processes with Tsallis entropy regularization, while \cite{gao2024reinforcement} considers a different use of the $q$-function for policy updates.

\subsubsection{Policy Gradient}
\label{sec:PG}

With the time--inhomogeneous $q$-function introduced, we are now ready to derive a policy gradient formula in our infinite-horizon jump-diffusion setting. We begin with two lemmas. 

\begin{lemma}
\label{lem:occupation-identity-inhom-shift}
Under the conditions in  Definition~\ref{def:beta_discount-inhom-shift}, for any measurable function $\varphi:[0,\infty)\times\mathbb{R}^d\to\mathbb{R}$ such that $\mathbb{E}\!\left[\int_t^\infty e^{-\beta (s-t)}|\varphi(s,\bm X_s^\pi)|\,\mathrm{d}s \right]<\infty$, we have
\begin{equation}\label{eq:occupation-identity-inhom-shift}
\mathbb{E}\Big[\int_{t}^{\infty} e^{-\beta (s-t)}\,\varphi(s,\bm X_s^\pi)\,\mathrm{d}s \Big]
=\int_{[t,\infty)\times\mathbb{R}^d} \varphi(s,\bm y)\,\mu^{\pi,t, \bm x}(\mathrm{d} s,\mathrm{d} \bm y)\,.
\end{equation}

\end{lemma}

This lemma extends \cite[Lemma~1]{zhao2023policy} to our time-inhomogeneous setting and follows from the occupation time formula.
 It states that any discounted pathwise reward $\varphi(s, \bm X_s^\pi)$ can be expressed as an integral of $\varphi$ with respect to the discounted time-state occupation measure $\mu^{\pi,t,\bm x}$. This identity will allow us to pass between expectations along trajectories and integrals over time-state space, which is crucial for writing performance differences in a concise integral form. 

\begin{lemma}
\label{lemma:gen_property_shift-final}
Let $\pi$ and $(\bm X_s^\pi)_{s\ge t}$ be as in
\cref{def:beta_discount-inhom-shift}, and let
$\varphi\in C^{1,2}([0,\infty)\times\mathbb R^d)$. 
Assume that $\mathbb E\!\int_t^\infty e^{-\beta(s-t)} \bigl( |\partial_s\varphi| +|\mathcal L^\pi\varphi| +\beta|\varphi| \bigr)(s,\bm X_s^\pi)\,\mathrm ds<\infty,$ and that $\lim_{T\to\infty} \mathbb E\!\left[e^{-\beta(T-t)}|\varphi(T,\bm X_T^\pi)| \right]=0.$
Then, for all $t\ge0$ and $\bm x\in\mathbb R^d$,
\begin{equation}
\label{eq:discounted-dynkin-shift-final}
\mathbb E\!\left[
\int_t^\infty e^{-\beta(s-t)}
\bigl(-\partial_s\varphi-\mathcal L^\pi\varphi+\beta\varphi\bigr)
(s,\bm X_s^\pi)\,\mathrm ds
\right] =
\varphi(t,\bm x),
\end{equation}
where $\mathcal L^\pi$ is defined in \eqref{eq:gen-pi}.
\end{lemma}

The proof follows the same argument as \cite[Lemma~8]{zhao2023policy}, with the addition of $\partial_s \varphi$ to account for the time inhomogeneity and a modified generator $\mathcal{L}^\pi$ to incorporate the jump terms. For brevity, the detailed proof is omitted. 
The discounted
integrability and transversality conditions allow functions of
polynomial growth, provided the state process has sufficiently high
moments whose growth rate is dominated by $\beta$.

Since Lemma~\ref{lemma:gen_property_shift-final} holds for any $\varphi \in C^{1,2}$ with certain property, we may replace $\varphi$ by $J(\cdot,\cdot;\hat\pi)$ that depends on a different stochastic control $\hat\pi$, whenever it is  regular enough. Therefore, quantities defined under $\hat \pi$ can be represented using the generator $\mathcal L^\pi$
 while expectations are taken along trajectories induced by $\pi$. This device is used below to compare $J(\cdot;\pi)$ and $J(\cdot;\hat\pi)$. 

\begin{theorem}[Policy gradient]
\label{thm:PG}
Let $\pi$ and $\hat{\pi}$ be two stochastic policies satisfying
the conditions of Definition~3.1, and let $\mu^{\hat\pi,t,\bm x}$ be the discounted occupation measure induced by $\hat\pi$ starting from $(t,\bm x)$.
Let $J(t,\bm x;\pi)$ be the value function under $\pi$.
Assume that $J(\cdot,\cdot;\pi)\in C^{1,2}([0,\infty)\times\mathbb R^d)$ is a classical solution of the PIDE \eqref{eq:policy-eval-PIDE}, and satisfies the conditions of Lemma~\ref{lemma:gen_property_shift-final} along the state process $\bm X^{\hat\pi}$.
Let $q(t,\bm x,\bm u;\pi)$ be the corresponding time-inhomogeneous $q$-function defined in \eqref{eq:q_function}. Then
\begin{equation}
\label{eq:perf_diff}
\resizebox{\linewidth}{!}{$
J(t,\bm x;\hat\pi)-J(t,\bm x;\pi)
=
\frac{1}{\beta}\,
\mathbb E_{(s,\bm X_s^{\hat\pi})\sim \beta\mu^{\hat\pi,t,\bm x},\,
\bm u\sim\hat\pi(\cdot\mid s,\bm X_s^{\hat\pi})}
\Big[
q(s,\bm X_s^{\hat\pi},\bm u;\pi)-\gamma\log\hat\pi(\bm u\mid s,\bm X_s^{\hat\pi})
\Big].
$}
\end{equation}

Now let $\{\pi_\theta(\bm u\mid t,\bm x)\}_{\theta\in\Theta}$ be a family of parameterized stochastic policies, and fix $\theta_0\in\Theta$. For each $\theta$, let $\mu^{\theta,t,\bm x}$ denote the discounted occupation measure of $(\bm X_s^{\pi_\theta})_{s\ge t}$ under $\pi_\theta$ with $\bm X_t=\bm x$, and let $q(\cdot;\pi_\theta)$ be the associated $q$-function. Assume additionally that all
$\pi_\theta$ in a neighborhood $U$ of $\theta_0$ are admissible and have
a common support, $\theta\mapsto
\pi_\theta(\bm u\mid t,\bm x)$ is differentiable. Assume further that the domination and weighted continuity conditions
stated in Section \suppref{appen:proof_lemma2} of the supplementary materials hold. Then, the mapping $\theta\mapsto J(t,\bm x;\pi_\theta)$ is
G\^ateaux differentiable at $\theta_0$, and
given by
\begin{equation}
\label{eq:grad_PG}
\resizebox{\linewidth}{!}{$
\nabla_\theta J(t,\bm x;\pi_\theta)\big|_{\theta=\theta_0}
=
\frac{1}{\beta}\,
\mathbb{E}_{(s,\bm y)\sim \beta\mu^{\theta_0,t,\bm x},\;
\bm u\sim \pi_{\theta_0}(\cdot\mid s,\bm y)}
\Big[
\nabla_\theta \log \pi_\theta(\bm u\mid s,\bm y)\big|_{\theta=\theta_0}\,
A_{\mathrm{ent}}(s,\bm y,\bm u;\theta_0)
\Big],
$}
\end{equation}
where the exploratory advantage is defined by
\begin{equation}
\label{eq:A_ent}
A_{\mathrm{ent}}(s,\bm y,\bm u;\theta_0)
:=
q(s,\bm y,\bm u;\pi_{\theta_0})
-\gamma\log \pi_{\theta_0}(\bm u\mid s,\bm y).
\end{equation}
\end{theorem}

The proof of the policy-gradient formula \eqref{eq:grad_PG} is given in Section~\suppref{appen:proof_of_thm2} in the supplementary materials.
This argument extends \cite[Theorem~3]{zhao2023policy}. In particular, when $\gamma = 0$, \eqref{eq:grad_PG} reduces to the classical policy-gradient formula \cite{sutton1999policy,montenegro2024learning}. 

Once Theorem~\ref{thm:PG} is established, we obtain an explicit representation of the policy gradient and can, in principle, learn the optimal policy via reinforcement learning.
However, the $q$-function is primarily a formal object: evaluating it may require derivatives such as $\nabla_{\bm x}J$ and $\nabla_{\bm x}^2J$, which is computationally expensive and numerically delicate.
Moreover, in model-free settings the SDE coefficients $(\bm b,\bm\sigma,\bm\alpha)$ are unknown, so the Hamiltonian terms cannot be computed directly. Therefore, practical implementations must rely on tractable approximations of $q$.
Two main approaches have been explored. First,  \cite{jia2023q,bo2024continuous}  approximate $q$ by neural networks and learn it via martingale properties, then update the policy via the Gibbs form implied by the $q$-function. Second,  \cite{zhao2023policy} relates $q$-function with $J$ without requiring derivatives, yielding an estimator akin to the generalized advantage estimation (GAE) \cite{tallec2019making}. 
Our approach follows the latter and leads to the next result under our setting with proof presented in Section~\suppref{appen:derivation_gae} in the supplementary materials.

\begin{lemma}[Approximation of $q$-function]
\label{lem:gae-approximation}
Fix $\beta>0$ and a stochastic policy $\pi$. Let $J(\cdot,\cdot;\pi)$ denote the corresponding discounted reward. 
Assume $J \in C^{1,2}([0,\infty)\times\mathbb R^d)$ and that there exists $\delta_0>0$ such that, for every $t\ge 0$ and every $\delta_t\in(0,\delta_0]$, $\mathbb E\Big[ \sup_{s\in[t,t+\delta_t)}
  \big(
    |J(s,\bm X_s^{\bm u};\pi)|
    + |\frac{\partial J}{\partial t}(s,\bm X_s^{\bm u};\pi)|
    + |\nabla_x J(s,\bm X_s^{\bm u};\pi)|
    + \|\nabla_x^2 J(s,\bm X_s^{\bm u};\pi)\|
  \big)
  \,\big|\,\\\mathcal F_t
\Big] < \infty,$
where $(\bm X^{\bm u}_s)_{s\geq t}$ is defined before \eqref{eq:Q_delta} with $\bm X_t^{\bm u} = \bm x$. 
Define the quantity 
\begin{equation}
\label{eq:gae_def_J}
\tilde q_{\delta_t}(t,\bm X_t^{\bm u}, \bm u;\pi)
:= \frac{1}{\delta_t}\big[
   f(t,\bm X_t^{\bm u},\bm u)\,\delta_t
   + e^{-\beta \delta_t}\,J\bigl(t+\delta_t, \bm X_{t+\delta_t}^{\bm u};\pi\bigr)
   - J\bigl(t, \bm X_{t}^{\bm u};\pi\bigr)
   \big].
\end{equation}
Then, as $\delta_t\to 0$,
\begin{equation}
\label{eq:gae-first-order_J}
\mathbb E\!\left[\,\tilde q_{\delta_t}(t,  \bm x, \bm u;\pi)\,\big|\,\mathcal F_t\right]
= q\bigl(t,\bm x, \bm u;\pi\bigr) + o(1).
\end{equation}
Therefore, $\tilde q_{\delta_t}(t, \bm x, \bm u;\pi)$ is a first-order asymptotically unbiased estimator of the $q$-function.
\end{lemma}

\begin{remark}
A convenient sufficient condition for the local integrability assumption in Lemma~\ref{lem:gae-approximation} is that $J$, $\partial_tJ$, $\nabla_xJ$, and $\nabla_x^2J$ have polynomial growth in $\bm x$, and that $\bm X^{\bm u}$ satisfies moment estimates of sufficiently high order on $[t,t+\delta_0]$. Such moment estimates follow under standard Lipschitz, growth, and jump-integrability conditions on the state coefficients; see, e.g., \cite[Proposition~2]{gao2024reinforcement} and \cite[Theorem~1.19]{oksendal2007applied}. 
\end{remark}

Therefore, when updating the actor using  \eqref{eq:grad_PG}--\eqref{eq:A_ent}, we replace $q$ by $\tilde q_{\delta_t}$, and correspondingly approximate $A_{\mathrm{ent}}(s,\bm y,\bm u;\theta_0)$ by $\tilde q_{\delta_t}(s,\bm y,\bm u;\pi_{\theta_0})-\gamma\log\pi_{\theta_0}(\bm u\mid s,\bm y)$.

\subsection{The Actor-Critic Scheme}
\label{sec:numerical_implementation}
We describe the proposed online actor-critic scheme for time-inhomogeneous jump diffusion control problems. The new scheme admits both model-free and model-based implementations. The actor update and the standard TD-based critic update require only sampled states, actions, rewards, and next states. The optional martingale-corrected critic update additionally uses the diffusion and jump coefficients, together with the corresponding observed or simulated noise increments. In our numerical experiments below, we use known simulators in order to compare the learned policies and value functions with analytical or high-accuracy benchmarks.

Fix a deterministic step size $\delta_t$ and consider the uniform grid
$0=t_0<t_1<\cdots<t_K$ with $t_k=k\,\delta_t$. We parameterize the actor and critic by neural networks, denoted by $\pi_\theta(\bm u \mid t, \bm x)$ and $V_\psi(t, \bm x)$, and update $(\theta, \psi)$ iteratively via policy gradient and policy evaluation. At each iteration, a minibatch of $L$ trajectories is used, which we denote by $\{\bm X_k^{(\ell)}\}_{k=0}^K$, $\ell = 1, \ldots, L$. Specifically, at time $t_k$, we sample an action $\bm u_k^{(\ell)}$ from the current policy $\pi_{\theta}(\cdot \mid t_k, \bm X_k^{(\ell)})$, apply the sampled action and observe the next state $\bm X_{k+1}^{(\ell)}$ from the environment. When a simulator is available, the transition may instead be generated using an Euler discretization of \eqref{eq:controlled-SDE-levy} with $\bm u_k^{(\ell)}$, see details in Section~\suppref{sec:discretization} of the supplementary materials.
The discounted reward $f_{k,\ell}$  accumulated over $[t_k, t_{k+1}]$ is also observed from the environment or approximated by $f_{k,\ell} :=  f(t_k,\bm X_k^{(\ell)},\bm u_k^{(\ell)}) \,\delta_t$ when the reward function $f$ is known.

\textbf{Critic.} To update the critic parameters $\psi$ in $V_\psi(t, \bm x)$, we construct the TD error \eqref{eq:ct_TD_error_J} (or the martingale-corrected TD error \eqref{eq:ct_martingale_corrected_TD}) along sampled trajectories using $\bar V_{\bar\psi}$ for the bootstrapped next-state target. The target network $\bar V_{\bar\psi}$ decouples the bootstrapping target from the online critic being optimized, thereby reducing target drift and improving stability during training, following the target-network idea used in soft actor-critic (SAC) \cite{haarnoja2018soft}. Specifically, for trajectory $\ell$ at time $t_k$, define
\begin{align}    \label{eq:TD_imple}
  & \delta_{\mathrm{TD},k}^{(\ell)}
:=\; f_{k,\ell}
-\gamma\log\pi_\theta^{(\ell)}(t_k)\delta_t
+e^{-\beta\delta_t}\,\bar V_{\bar\psi}\!\bigl(t_{k+1},\bm X_{k+1}^{(\ell)}\bigr)
- V_\psi\!\bigl(t_k,\bm X_k^{(\ell)}\bigr),\\
\label{eq:martingale_corrected_TD-discrete_impl}
&    \tilde\delta_{\mathrm{TD},k}^{(\ell)}
:=\; \delta_{\mathrm{TD},k}^{(\ell)}
-\big( \nabla_x V_\psi\!\bigl(t_k,\bm X_k^{(\ell)}\bigr)^\top \bm\sigma\!\bigl(t_k,\bm X_k^{(\ell)},\bm u_k^{(\ell)}\bigr)\big) \Delta\bm W_k^{(\ell)} -\Big(
\sum_{i=N_{t_k}^{(\ell)}+1}^{N_{t_{k+1}}^{(\ell)}}
\\
\,\big[
V_\psi\!&\bigl(t_k,\bm X_k^{(\ell)} +\bm\alpha(t_k,\bm X_k^{(\ell)},\bm u_k^{(\ell)},\bm z_i^{(\ell)})\bigr) - V_\psi\!\bigl(t_k,\bm X_k^{(\ell)}\bigr)
\big]
-\delta_t V_{\text{non}}\!\bigl(t_k,\bm X_k^{(\ell)},\bm u_k^{(\ell)}\bigr)
\Big),\nonumber
\end{align}
where $\Delta \bm W_k^{(\ell)}:=\bm W_{t_{k+1}}^{(\ell)}-\bm W_{t_k}^{(\ell)}\sim \mathcal{N}(0,\delta_t \bm I_d)$,
 $N_{t_k}^{(\ell)}$ counts the number of jumps on the trajectory $\ell$,
$\{\bm z_i^{(\ell)}\}_{i\ge1}$ are the corresponding jump sizes. The quantity $V_{\mathrm{non}}(t,\bm x,\bm u)
:= \int_{\mathbb R^d}
\big[
V_\psi\!\bigl(t, \bm{x}+\bm\alpha(t,\bm{x}, \bm{u},\bm z)\bigr)-V_\psi(t,\bm{x})
\big]\nu(\mathrm d\bm z)$
is the instantaneous compensator rate of the jump increments of
$V_\psi$. Indeed, since
$\widetilde N(\mathrm dt,\mathrm d\bm z)
=N(\mathrm dt,\mathrm d\bm z)-\nu(\mathrm d\bm z)\mathrm dt$,
the jump martingale increment over $[t_k,t_{k+1}]$ is approximated by
the realized jump sum minus
$\delta t\,V_{\mathrm{non}}(t_k,\bm X_k,\bm u_k)$. Under Poisson jumps,  $\nu(\mathrm d\bm z) = \sum_{i=1}^{d}\lambda_i \delta_{\bm e_i}(\mathrm d\bm z)$,
the integral is evaluated exactly as $V_{\mathrm{non}}(t,\bm x,\bm u) = \sum_{i=1}^{d}\lambda_i \left[ V_\psi\bigl( t,\bm x+\bm\alpha(t,\bm x,\bm u,\bm e_i) \bigr) - V_\psi(t,\bm x) \right]$. For a general known finite-activity L\'evy measure, approximation and
implementation details are provided in Section~\suppref{sec:mart_disc} (supplementary materials). Note that this term is required only for the martingale-corrected TD update.

Given a minibatch of $L$ trajectories and $K_{\mathrm{critic}}$ time steps, we update $\psi$ by minimizing the empirical mean-squared (martingale-corrected) TD error 
\begin{equation}
\label{eq:loss_critic_impl}
L_{\mathrm{critic}}(\psi)
=  \frac{1}{L K_{\mathrm{critic}}}
  \sum_{\ell=1}^L \sum_{k=0}^{K_{\mathrm{critic}}-1}
  (\delta_{\text{TD}, k}^{(\ell)})^2\, \text{ or } \frac{1}{L K_{\mathrm{critic}}}
  \sum_{\ell=1}^L \sum_{k=0}^{K_{\mathrm{critic}}-1}
  (\tilde\delta_{\text{TD}, k}^{(\ell)})^2.
\end{equation}
When minimizing \eqref{eq:loss_critic_impl} with the usual one-step TD error, we apply semi-gradient: the bootstrapped target $\bar V_{\bar\psi}\!\bigl(t_{k+1},\bm X_{k+1}^{(\ell)}\bigr)$ is held fixed when differentiating with respect to $\psi$, which is consistent with the martingale orthogonality condition in \cite{jia2022policy,jia2022PE}.
The target critic parameters are then updated by Polyak averaging $\bar\psi \leftarrow \rho_c\,\bar\psi + (1-\rho_c)\,\psi,
\, \rho_c\in(0,1)$, where $\rho_c$ is a prescribed averaging weight \cite{haarnoja2018soft}.
The martingale correction is an optional variance-reduction device for the critic update and requires additional information about the model coefficients and noise increments. To isolate its variance-reduction effect, the supplementary materials Section~S.3.3 compares the usual and martingale-corrected TD errors under the same numerical setting. This comparison demonstrates the effect of removing the martingale fluctuation rather than to provide a model-free performance comparison, since the martingale-corrected TD error uses additional information. The two supplementary comparisons indeed confirm that the martingale correction substantially reduces the TD-residual variance and that this variance reduction translates into improved value estimation and policy learning.

\textbf{Actor.} We parameterize the stochastic policy $\pi_\theta$ as a conditional normalizing flow \cite{brahmanage2023flowpg,papamakarios2021normalizing}. Specifically, the policy is defined as the pushforward of a Gaussian distribution $\mathcal{N}_\theta$\footnote{Even in the unregularized case ($\gamma = 0$), where the optimal control is deterministic, we still adopt a Gaussian policy parameterization. This choice is supported by \cite{montenegro2024learning,wang2020reinforcement}, which establish the convergence of policy-gradient methods with stochastic policies and show that, in terms of solvability, they are equivalent to the corresponding standard control problem.} 
\begin{equation}
\label{eq:policy_parameterized_impl}
\bm z_0 \sim \mathcal{N}_\theta(\cdot  \mid t,\bm x)
= \mathcal{N}\big(\bar{\bm \mu}_\theta(t,\bm x),\,\mathrm{Std}_{\theta}^2(t,\bm x)\big),
\end{equation}
through a learnable invertible normalizing flow $F_\theta(\cdot;  t, \bm x)$, followed by an optional differentiable squashing map $S$ (e.g., a \texttt{sigmoid} or \texttt{tanh}) that enforces control constraints when needed. The log-density of the resulting control sample is computed exactly by the change-of-variables formula,
\begin{equation}
\begin{aligned}
\label{eq:logpi_flow_general}
\log \pi_\theta(\bm u\mid t, \bm x)
& =
\log p_{{\mathcal N}_{\theta}}\!\bigl(\bm z_0 \mid \bar{\bm\mu}_\theta(t, \bm x), \mathrm{Std}_\theta(t,\bm x)\bigr)
-
 \log \bigl|\det J_{F}(\bm z_0; t, \bm x)\bigr|\\
&\quad -
\log \bigl|\det J_{S}(F_\theta(\bm z_0))\bigr|,
\end{aligned}
\end{equation}
where $J_{F}$ and $J_S$ denote the Jacobians of the flow $F$ and the squashing map $S$, respectively. This construction defines a flexible stochastic policy class that includes Gaussian policies as a special case\footnote{When the Hamiltonian $\mathscr{H}$ is quadratic in $\bm u$, the optimal randomized policy $\pi^*(\bm u\mid t,\bm x)\propto \exp(\mathscr{H} (t,\bm x,\bm u)/\gamma)$ is Gaussian; if the control domain is unconstrained, we may also omit the squashing map, so that only the first term in \eqref{eq:logpi_flow_general} remains.}, while retaining exact likelihood evaluation required for entropy regularization and policy-gradient optimization. The simultaneous parameterization of the Gaussian base distribution and the flow map is generally non-identifiable. This does not affect the method, since only the induced policy distribution is relevant. A fixed standard Gaussian base could alternatively be used; we retain a learnable Gaussian base because it includes learnable Gaussian policies when the flow is the identity.

Our implementation of $F_\theta$ uses rational-quadratic spline coupling layers~\cite{durkan2019neural}. Their coupling structure yields a triangular Jacobian with respect to the latent input, so $\log|\det J_F|$ is evaluated as a sum of one-dimensional spline derivatives rather than by forming a dense Jacobian matrix. Gradients with respect to $\theta$ are computed by automatic differentiation through the resulting scalar log-density. Further implementation details are provided in Appendix~\ref{sec:numericalsupplements}. When constraints are imposed, the present implementation considers box-type constraints, for which $S$ is chosen as a coordinatewise \texttt{sigmoid} or \texttt{tanh} map. Its Jacobian is diagonal, so $\log|\det J_S|$ is computed as a sum of coordinatewise log-derivatives. For more general constraint sets, a problem-specific invertible reparameterization may be required.

In implementation, samples of $\pi_\theta(\cdot \mid t_k, \bm X_k^{(\ell)})$ are generated using
\begin{align}
    \label{eq:flow_transform}\bm z_{0,k}^{(\ell)}
&= \bar{\bm \mu}_\theta\big(t_k, \bm X_k^{(\ell)}\big) + \mathrm{Std}_{\theta}\big(t_k, \bm X_k^{(\ell)}\big)\odot  \bm \varepsilon_k^{(\ell)},\,
\bm\varepsilon_k^{(\ell)} \sim \mathcal{N}(\bm 0,\bm I_m), \\
    \bm z_{F,k}^{(\ell)} &= F_\theta \bigl(\bm z_{0,k}^{(\ell)} ; t_k, \bm X_k^{(\ell)}\bigr),\quad \bm u_k^{(\ell)} = S(\bm z_{F,k}^{(\ell)})\in [\bm u_{\min}, \bm u_{\max}]^m. \nonumber
\end{align}
To update the actor parameters $\theta$, we form the one-step advantage estimator using \eqref{eq:A_ent} and Lemma~\ref{lem:gae-approximation}:
\begin{equation}
\label{eq:adv_rate_estimator_impl}
\hat A_k^{(\ell)}
:= \frac{1}{\delta_t}
   \Big(
     f_{k,\ell}
     + e^{-\beta \delta_t}\,V_\psi\bigl(t_{k+1}, \bm X_{k+1}^{(\ell)}\bigr)
     - V_\psi\bigl(t_k, \bm X_k^{(\ell)}\bigr)
   \Big) - \gamma\,\log\pi_\theta^{(\ell)}(t_k),
\end{equation}
where $\log\pi_\theta^{(\ell)}(t_k) := \log \pi_\theta(\bm u_k^{(\ell)} \mid t_k, \bm X_{k}^{(\ell)})$ using \eqref{eq:logpi_flow_general}. Given a minibatch of $L$ trajectories and $K_{\mathrm{actor}}$ time steps, we update $\theta$ by minimizing the policy-gradient surrogate
\begin{equation}
\label{eq:actor_loss_impl}
L_{\mathrm{actor}}(\theta)
= -\,\frac{1}{\beta L K_{\mathrm{actor}}}
  \sum_{\ell=1}^L \sum_{k=0}^{K_{\mathrm{actor}}-1}
  \log\pi_\theta^{(\ell)}(t_k)\,
  \mathrm{stopgrad}\!\bigl(\hat A_k^{(\ell)}\bigr),
\end{equation}
where $\mathrm{stopgrad}(\cdot)$ indicates that $\hat A_k^{(\ell)}$ is treated as constant when differentiating with respect to $\theta$, consistent with Theorem~\ref{thm:PG}.

Note that the one-step estimator in \eqref{eq:adv_rate_estimator_impl} may have large variance for small $\delta_t$, since the terms in parentheses contain Brownian and compensated-jump martingale increments of order $\sqrt{\delta_t}$ which become $\delta_t^{-1/2}$ after division by $\delta_t$. Thus we reduce this variance by averaging over $L$ trajectories and $K_{\mathrm{actor}}$ time steps, cf. \eqref{eq:actor_loss_impl}. When the model coefficients and the noise increments are available, one may further subtract the corresponding martingale correction from the actor estimator, analogously to the optional martingale-corrected TD error; i.e. subtracting $(\delta_{\mathrm{TD},k}^{(\ell)}-\tilde{\delta}_{\mathrm{TD},k}^{(\ell)})/\delta_t$ from $\hat A_k^{(\ell)}$.

Combining \eqref{eq:loss_critic_impl} and \eqref{eq:actor_loss_impl} yields the online time-inhomogeneous actor-critic procedure summarized in Alg.~\ref{alg:ours}. We note that the underlying control problem is infinite-horizon: the finite sum over $k$ corresponds to an optimization window used for stability and variance reduction, and does not impose a finite terminal time.

{\small
\begin{algorithm}[t]
\caption{Time-inhomogeneous actor-critic for infinite-horizon jump-diffusion}
\label{alg:ours}
\begin{algorithmic}[1]
\Require Step size $\delta_t$; discount $\beta$; entropy weight $\gamma$; number of time points $K$; update periods $K_{\mathrm{critic}}, K_{\mathrm{actor}}$; iterations $N_{\mathrm{itr}}$; minibatch $L$;
value net $V_\psi(t,\bm x)$; target net $\bar V_{\bar\psi}(t,\bm x)$; Polyak coefficient $\rho_c\in(0,1)$; stochastic policy $\pi_\theta(\cdot\mid t,\bm x)$.
\For{$\mathrm{it}=1$ to $N_{\mathrm{itr}}$}
  \State Set initial time $t \gets 0$ and initial states $\bm X_0 \gets \bm X_{\mathrm{init}}$.
  \State $V \gets V_\psi(t,\bm X)$.
  \State $\mathcal{L}_{\mathrm{critic}}\gets 0,\; n_{\mathrm{critic}}\gets 0;\quad
         \mathcal{L}_{\mathrm{actor}}\gets 0,\; n_{\mathrm{actor}}\gets 0.$
  \For{$k=0$ to $K-1$}
\State Sample $\bm u_k^{(\ell)}$ according to  \eqref{eq:flow_transform}. 

    \State Calculate the log-density $\log\pi_\theta^{(\ell)}(t_k)$ via \eqref{eq:logpi_flow_general}.

    \State Observe the next state $\bm X_{k+1}^{(\ell)}$ from the environment or evolve from $\bm X_k^{(\ell)}$ via Euler scheme.

    \State Compute the TD error $\delta_{\mathrm{TD},k}^{(\ell)}$ using \eqref{eq:TD_imple} or \eqref{eq:martingale_corrected_TD-discrete_impl}.

    \State $\mathcal{L}_{\mathrm{critic}} \gets \mathcal{L}_{\mathrm{critic}} + \sum_{\ell = 1}^L \|\delta_{\text{TD},k}^{(\ell)}\|_2^2$; \quad $n_{\mathrm{critic}} \gets n_{\mathrm{critic}} + 1$.

    \If{$n_{\mathrm{critic}} \bmod K_{\mathrm{critic}} = 0$}
      \State Update $\psi$ by one optimizer step on $\mathcal{L}_{\mathrm{critic}}/LK_{\mathrm{critic}}$. \Comment{critic objective \eqref{eq:loss_critic_impl}}
      \State $\bar\psi \gets \rho_c\,\bar\psi + (1-\rho_c)\,\psi$.
      \State $\mathcal{L}_{\mathrm{critic}}\gets 0$; \quad $n_{\mathrm{critic}}\gets 0$.
    \EndIf

\State Compute the GAE estimator $\hat A_k^{(\ell)}$ according to \eqref{eq:adv_rate_estimator_impl}

    \State $\mathcal{L}_{\mathrm{actor}} \gets \mathcal{L}_{\mathrm{actor}} - \beta^{-1}\sum_{\ell=1}^L \log\pi_\theta^{(\ell)}(t_k) \; \mathrm{stopgrad}(\hat A_k^{(\ell)})$;  \,\, $n_{\mathrm{actor}} \gets n_{\mathrm{actor}} + 1$.
    
    \If{$n_{\mathrm{actor}} \bmod K_{\mathrm{actor}} = 0$}
      \State Update $\theta$ by one optimizer step on $\mathcal{L}_{\mathrm{actor}}/LK_{\mathrm{actor}}$.\Comment{actor objective \eqref{eq:actor_loss_impl}}
      \State $\mathcal{L}_{\mathrm{actor}}\gets 0$; \quad $n_{\mathrm{actor}}\gets 0$.
    \EndIf

  \EndFor
\EndFor
\end{algorithmic}
\end{algorithm}
}

\section{Numerical Experiments}
\label{sec:analytical_examples}
In this section, we illustrate the proposed actor-critic framework through a set of representative numerical examples.
We consider three problems of increasing complexity: a linear-quadratic (LQ) control problem with jump diffusion (Section \ref{sec:lq_unified}), the Merton portfolio optimization problem (Section \ref{sec:merton}), and a multi-agent portfolio game (Section \ref{sec:multi-agent}).
These examples are chosen to demonstrate the flexibility of the method across settings with known analytical structure, nonlinear dynamics, and strategic interactions, as well as to assess its empirical stability and performance in time-inhomogeneous jump-diffusion environments.

\smallskip
\noindent\textbf{Metrics.}
We assess the learned actor and critic networks using trajectory-level metrics. For each experiment, we analyze the learned (i) state trajectory $\hat{\bm X}_t$, (ii) value function $\hat V(t,\hat{\bm X}_t)$, and (iii) control process (the feedback control $\hat{\bm u}(t,\hat{\bm X}_t)$ when $\gamma = 0$, or the mean of $\hat{\bm u}_t \sim \hat\pi(\cdot\mid t,\hat{\bm X}_t)$ in the exploratory case $\gamma>0$).

When a benchmark solution $(\bm u^*,V,\bm X)$ is available, we report time-averaged relative mean-square errors (RMSEs) on $[0,T_{\mathrm{eval}}]$:
\begin{equation}
\label{eq:error_X_V_rel}
\mathcal{E}_X(T_{\mathrm{eval}})
:=
\frac{\int_0^{T_{\mathrm{eval}}}\!\|\hat{\bm X}_t-\bm X_t\|^2\,\mathrm{d}t}
{\int_0^{T_{\mathrm{eval}}}\!\|\bm X_t\|^2\,\mathrm{d}t+\varepsilon_X},
\qquad
\mathcal{E}_V(T_{\mathrm{eval}})
:=
\frac{\int_0^{T_{\mathrm{eval}}}\!|\hat V(t,\hat{\bm X}_t)-V(t,\bm X_t)|^2\,\mathrm{d}t}
{\int_0^{T_{\mathrm{eval}}}\!|V(t,\bm X_t)|^2\,\mathrm{d}t+\varepsilon_V}.
\end{equation}
For the learned control, we use the RMSE when $\gamma=0$, and a distributional discrepancy when $\gamma>0$:
\begin{equation}
\label{eq:error_u_rel}
\mathcal{E}_u(T_{\mathrm{eval}})
:=
\begin{cases}
\dfrac{\int_0^{T_{\mathrm{eval}}}\!
\|\hat{\bm u}(t,\hat{\bm X}_t)-\bm u^*(t,\bm X_t)\|^2\,\mathrm{d}t}
{\int_0^{T_{\mathrm{eval}}}\!\|\bm u^*(t,\bm X_t)\|^2\,\mathrm{d}t+\varepsilon_u},
& \gamma=0,\\[14pt]
\dfrac{1}{T_{\mathrm{eval}}}\displaystyle\int_0^{T_{\mathrm{eval}}}
\mathrm{KL}\big(\pi^*(\cdot\mid t,\hat{\bm X}_t)\,\big\|\,\hat\pi(\cdot\mid t,\hat{\bm X}_t)\big)\,\mathrm{d}t,
& \gamma>0,
\end{cases}
\end{equation}
since for $\gamma>0$, both the benchmark and learned controls are stochastic policies, denoted by $\pi^*(\cdot\mid t,\bm x)$ and $\hat\pi(\cdot\mid t,\bm x)$.
The constants $\varepsilon_X,\varepsilon_V,\varepsilon_u>0$ are small stabilizers included to avoid division by zero.

\paragraph{Poisson jump specification}
In the experiments, we consider a discrete L\'evy measure corresponding to Poisson jumps: $\nu(\mathrm{d}\bm z)=\sum_{i=1}^{d}\lambda_i\,\delta_{\bm e_i}(\mathrm{d}\bm z),$ where $\lambda_i>0$ and $\bm e_i$ is the $i$-th canonical basis vector in $\mathbb R^d$. Under this specification, the jump measure is represented by independent Poisson processes $N_t^{(i)}$ with rates $\lambda_i$, so that $N(\mathrm{d}t,\bm e_i)=\mathrm{d}N_t^{(i)}$. Let $M_t^{(i)}:=N_t^{(i)}-\lambda_i t$ be the compensated Poisson process, namely $\mathrm{d}M_t^{(i)}=\mathrm{d}N_t^{(i)}-\lambda_i\,\mathrm{d}t$. Then the jump term in \eqref{eq:controlled-SDE-levy} becomes $\int_{\mathbb{R}^{d}} \bm\alpha\bigl(t,\bm X^\pi_{t-},\bm u_t,\bm z\bigr)\,\tilde N(\mathrm{d}t,\mathrm{d}\bm z)=\sum_{i=1}^{d}\bm\alpha\bigl(t,\bm X^\pi_{t-},\bm u_t,\bm e_i\bigr)\,\mathrm{d}M_t^{(i)}.$ Accordingly, the nonlocal integral term in the Hamiltonian \eqref{eq:Hamiltonian-general} and the generator $\mathcal{L}^{\bm u}$ in \eqref{eq:generator-general-levy} reduces to $\sum_{i=1}^{d}\lambda_i\Bigl[V\bigl(t,\bm x+\bm\alpha(t,\bm x,\bm u,\bm e_i)\bigr)$ $-V(t,\bm x)-\bm\alpha(t,\bm x,\bm u,\bm e_i)\cdot \nabla_{\bm x}V(t,\bm x)\Bigr].$

In each of the following sections, we (i) specify the control or game formulation, (ii) give the analytical solution when available, or present how benchmark solutions are obtained,  and (iii) present the numerical results. All implementation details and model parameters are provided in Appendix~\ref{sec:numericalsupplements}. Some benchmark derivations are standard in the stochastic control literature, and are included in the supplementary materials for completeness. All experiments are implemented in \texttt{PyTorch} and run on an NVIDIA RTX \texttt{4090} GPU. The code is available upon request and will be made public upon publication.

\subsection{Linear-Quadratic Control with Jump Diffusions}
\label{sec:lq_unified}
We first consider a $d$-dimensional state $\bm X_t\in\mathbb R^d$ and a $m$-dimensional control $\bm u_t\in\mathbb R^m$, governed by the controlled jump-diffusion
\begin{equation}
\label{eq:lq_sde_unified}
\mathrm{d}\bm X_t
= \bm B(t)\,\bm u_t\,\mathrm{d} t
  + \bm\Sigma(t)\,\mathrm{d}\bm W_t
  + \sum_{i=1}^d \alpha_i(t)\,\bm e_i\,\mathrm{d}M_t^{(i)},
\quad t\ge 0,
\end{equation}
where $\bm B(t)\in\mathbb R^{d\times m}$, $\bm\Sigma(t)\in\mathbb R^{d\times d}$, $\bm \alpha(t):=(\alpha_1(t),\dots,\alpha_d(t))^\top\in\mathbb R^d$, and $\bm e_i$ is the $i$th canonical basis vector in $\mathbb R^d$.
The running reward is quadratic $f(t,\bm x,\bm u) = -(\bm u^\top \bm R(t)\,\bm u + \bm x^\top \bm Q(t)\,\bm x )$, where $\bm R(t)\in\mathbb S^{m}$ and $\bm Q(t)\in\mathbb S^{d}$ are prescribed positive definite model parameters. Consequently, $\bm R(t)^{-1}$ is well defined.

Recall that the value function satisfies \eqref{eq:HJB-hard} for the standard control $(\gamma = 0)$ and \eqref{eq:soft-HJB-opt-general} under entropy-regularization, with the integral term replaced by $\sum_{i=1}^d \lambda_i(t)\big(\varphi$ $(t,\bm x +\alpha_i(t)\bm e_i)-\varphi(t,\bm x)-\alpha_i(t)\,\partial_{x_i}\varphi(t,\bm x)\big)$. Therefore, the optimal policy and value function satisfy
\begin{equation}
\label{eq:lq_pi_star_convergent_merged}
\pi^*(\bm u\mid t,\bm x)
=
\mathcal N\!\left(
\bm R(t)^{-1}\bm B(t)^\top \bm H(t)\bm x,\;
\frac{\gamma}{2}\bm R(t)^{-1}
\right), \; V(t,\bm x)=\bm x^\top \bm H(t)\bm x+g_\gamma(t) \,,
\end{equation}
where $\bm H(t)\in\mathbb S^{d}$ and $g_\gamma(t)\in\mathbb R$ solve
\begin{equation}
\label{eq:lq_riccati_scalar_inhomo}
\begin{aligned}
\bm H'(t)
&=
\beta \bm H(t)+\bm Q(t)-\bm H(t)\bm B(t)\bm R(t)^{-1}\bm B(t)^\top \bm H(t),\\
g_\gamma'(t)
&=
\beta g_\gamma(t)
-\operatorname{Tr}\!\bigl(\bm\Sigma(t)\bm\Sigma(t)^\top \bm H(t)\bigr)
-\operatorname{Tr}\!\Bigl(\bm\Lambda(t)\operatorname{diag}(\bm\alpha(t))\bm H(t)\operatorname{diag}(\bm\alpha(t))\Bigr)
-c_\gamma(t),
\end{aligned}
\end{equation}
with $c_\gamma(t)=\frac{\gamma}{2}\bigl(m\log(\pi\gamma)-\log\det \bm R(t)\bigr), \,\bm\Lambda(t):=\operatorname{diag}(\lambda_i(t)).$ The proof for analytical solution, together with the boundary conditions for the ODE system~\eqref{eq:lq_riccati_scalar_inhomo}, are presented in Section~\suppref{sec:lq_closed_form} (supplementary materials). In the classical case $\gamma=0$, by standard derivation, the optimal stochastic policy degenerates to the feedback control
\begin{equation}
\label{eq:lq_u_star_convergent_merged}
\bm u^*(t,\bm x)=\bm R(t)^{-1}\bm B(t)^\top \bm H(t)\bm x.
\end{equation}

\begin{figure}[htbp]
  \centering
  \begin{subfigure}{0.95\textwidth}
    \centering
\includegraphics[width=\linewidth]{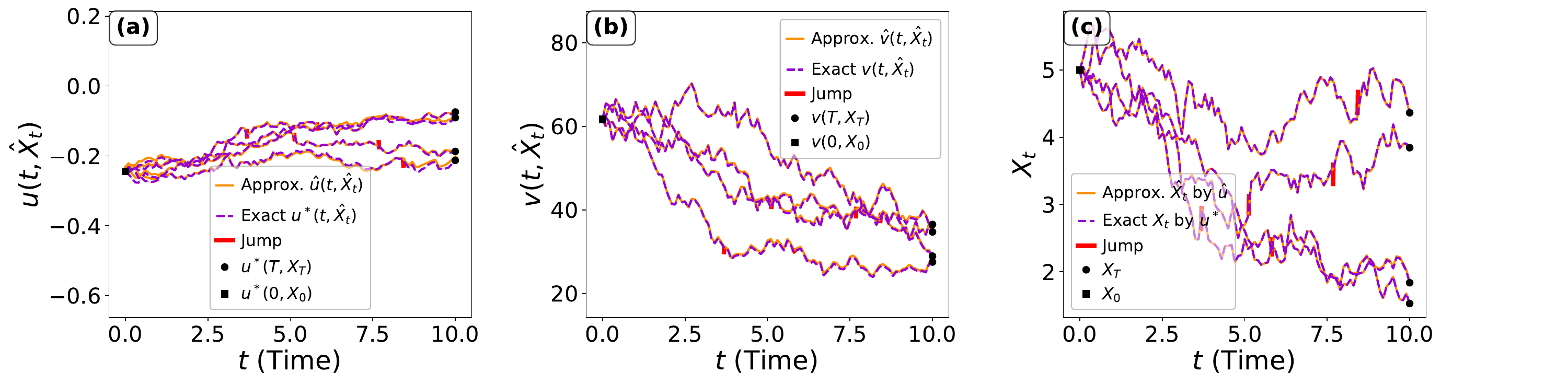}
\includegraphics[width=\linewidth]{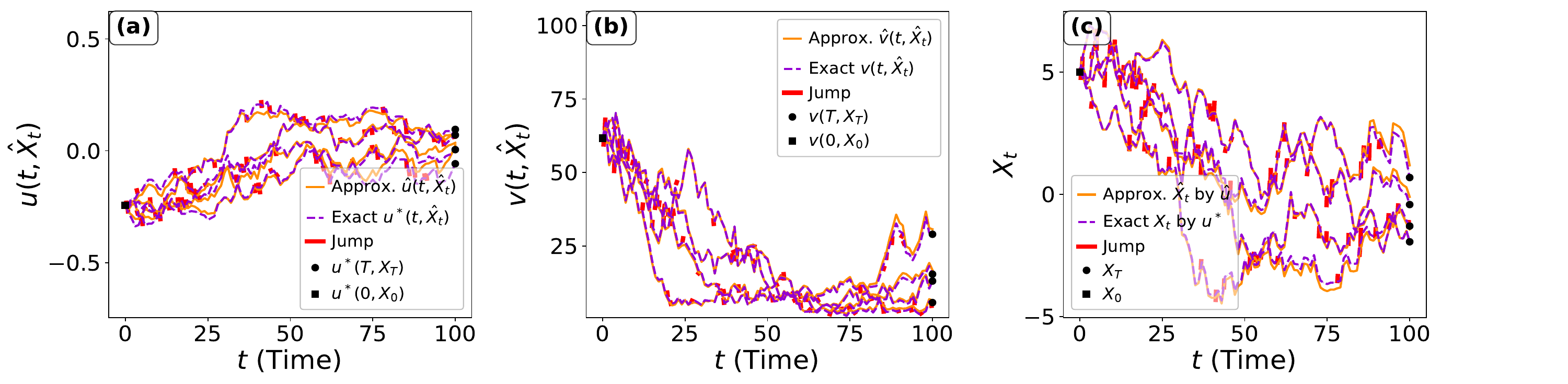}
  \end{subfigure}
  \caption{Predicted (a) optimal control, (b) value function, and (c) state trajectory for the standard LQ problem with $d=5$ on horizon $T=10$ (top) and $T=100$ (bottom). Parameters: $\gamma=0$, $\bm B= 0.5 \bm I_d$, $\bm\Sigma=0.3\bm I_d$, $\bm R = 5\bm I_d$, $\bm Q = 0.5\bm I_d$,
  $\lambda_i = 0.2+\frac{i-1}{d-1}(0.3-0.2)$ and $\alpha_i = 0.3-\frac{i-1}{d-1}(0.3-0.2)$ for $i=1,\ldots,d$.
  }
  \label{fig:standard_LQ_results_D5}
\end{figure}

\subsubsection{Time-Homogeneous Case}
We start with the time-homogeneous case, i.e.,
$\bm B(t)\equiv \bm B$, $\bm\Sigma(t)\equiv \bm\Sigma$, $\bm\alpha(t)\equiv \bm\alpha$, $\bm\Lambda(t)\equiv \bm\Lambda$, $\bm R(t)\equiv \bm R$ and $\bm Q(t)\equiv \bm Q$. In this case,
there exists a stationary pair $(\bm H, g_\gamma)$ that solves:

\begin{equation}
\begin{aligned}
\label{eq:ARE_homo_entropy_merged}
\bm 0 &=\beta \bm H+\bm Q-\bm H\bm B\bm R^{-1}\bm B^\top \bm H, \\
\beta g_\gamma
&=
\operatorname{Tr}\!\bigl(\bm{\Sigma}\bm{\Sigma}^\top \bm{H}\bigr)
+\operatorname{Tr}\!\Bigl(\bm{\Lambda}\,\operatorname{diag}(\bm{\alpha})\,\bm{H}\,\operatorname{diag}(\bm{\alpha})\Bigr)
+\frac{\gamma}{2}\bigl(m\log(\pi\gamma)-\log\det \bm{R}\bigr).
\end{aligned}
\end{equation}
The optimal stochastic policy remains Gaussian with mean $\bm R^{-1}\bm B^\top \bm H\bm x$ and covariance $\frac{\gamma}{2}\bm R^{-1}$, and the standard case follows by setting $c_\gamma=0$.

\smallskip
\noindent\textbf{Standard LQ problem ($\gamma=0$).}
We first solve a 5-dimensional LQ control problem, and train the actor and critic networks $(\pi_\theta, V_\psi)$ using Algorithm~\ref{alg:ours} for $N_{\mathrm{itr}}=1{,}000$ iterations, with horizons $T\in\{10,100\}$.  
We then freeze all networks and generate the approximated state trajectories $\hat {\bm{X}}_t$, value function $\hat V$ and the feedback control $\hat {\bm u}$ using step size $\delta_t=0.01$.
Figure~\ref{fig:standard_LQ_results_D5} shows that the learned values are consistent with the analytical solution, and remain numerically stable even over the long horizon $T=100$.

\smallskip
\noindent\textbf{Entropy-regularized LQ problem ($\gamma>0$).}
We next consider the entropy-regularized variant with $\gamma=0.05$ under the same problem setup and evaluation procedure. For the LQ problem, the Hamiltonian $\mathscr{H}$ is quadratic in $\bm u$, so the optimal entropy-regularized policy is Gaussian. Accordingly, when parameterizing the actor we may treat the flow map as the identity and use only the Gaussian policy in \eqref{eq:policy_parameterized_impl}. Equivalently, sampling actions reduces to setting $\bm u$ to the base variable in \eqref{eq:flow_transform}, i.e., $\bm u=\bm z_0$.
Figure~\ref{fig:exploratory_LQ_results_D5} reports the mean of stochastic control as well as  the approximated state and value trajectories, showing that Alg.~\ref{alg:ours} remains numerically stable and accurate under exploration.

\begin{figure}[htbp]
  \centering
  \begin{subfigure}{0.95\textwidth}
    \centering
    \includegraphics[width=\linewidth]{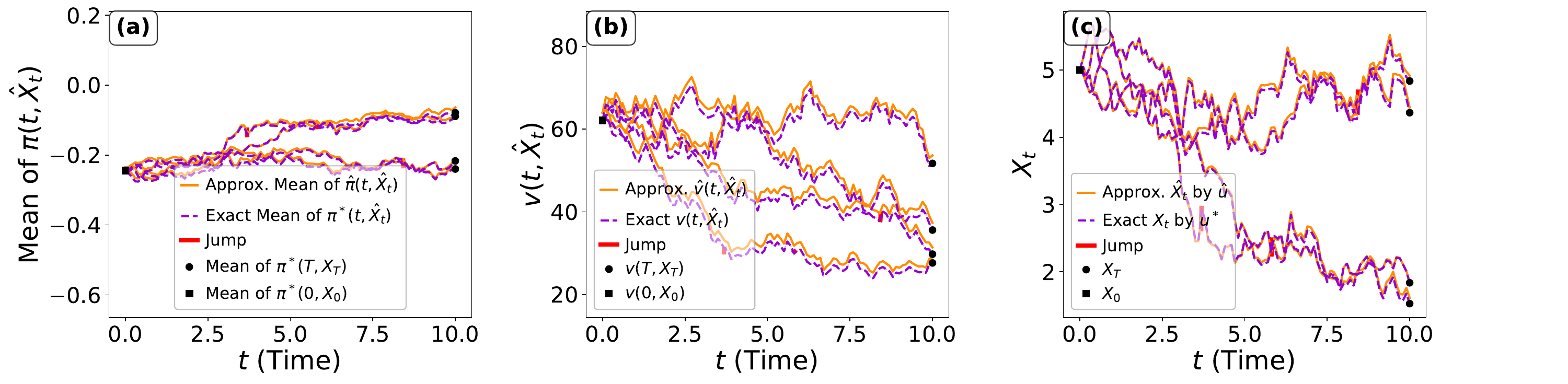}
  \end{subfigure}
  \caption{Predicted (a) mean of stochastic policy, (b) value function, and (c) state trajectory for the entropy-regularized LQ problem with $d=5$ and $\gamma=0.05$. All remaining parameters and the evaluation protocol are the same as in Fig.~\ref{fig:standard_LQ_results_D5}, with horizon $T=10$.}
  \label{fig:exploratory_LQ_results_D5}
\end{figure}

To test scalability in the state dimension, we repeat the entropy-regularized experiment ($\gamma=0.05$) for
$d\in\{1,5,20,50\}$. Table~\ref{tab:lq_training_error_vs_d} reports the  RMSE for the value and control over three seeds. The value error $\mathcal{E}_V$ stays small across dimensions,
whereas the control error $\mathcal{E}_u$ increases linearly with $d$, consistent with the greater difficulty of learning high-dimensional feedback under exploration noise.

\begin{table}[htbp]
\centering
\caption{Training errors for the entropy-regularized ($\gamma=0.05$) LQ problem across state dimensions $d$. All runs use $N_{\mathrm{itr}}=1{,}000$ iterations and results are averaged over three random seeds.}

\label{tab:lq_training_error_vs_d}
\small
\begin{tabular}{lcccc}
\toprule
& $d=1$ & $d=5$ & $d=20$ & $d=50$ \\
\midrule
$\mathcal{E}_{u}$ & $0.1441$ & $0.3992$ & $1.9205$ & $4.9366$ \\
$\mathcal{E}_{V}$ & $0.0034$ & $0.0036$ & $0.0046$ & $0.0038$ \\
\bottomrule
\end{tabular}
\end{table}

\subsubsection{Time-Inhomogeneous Case}
\textbf{Convergent coefficients.} We next consider a time-inhomogeneous LQ problem such that as $t \rightarrow \infty$,
\begin{equation*}
\bm B(t)\to \bm B_\infty,\,
\bm \Sigma(t)\to \bm \Sigma_\infty,\,
\bm\alpha(t)\to \bm\alpha_\infty,\,
\bm \Lambda(t) \to \bm \Lambda_\infty,
\, \bm Q(t)\to \bm Q_\infty,\,
\bm R(t)\to \bm R_\infty,\,
\end{equation*}
with sufficiently fast convergence so that the discounted reward is well defined.

The computation of the benchmark solution is less direct than in the time-homogeneous case, since the ODE system \eqref{eq:lq_riccati_scalar_inhomo} does not come with an explicit terminal boundary condition. For the present choice of convergent coefficients, the terminal boundary can be approximated by introducing a sufficiently large terminal time \(T_\infty\) and using the limiting stationary solution there as an approximate boundary condition. In our implementation, we take $T_\infty = 3T$. The control process is considered on the interval $[0,T]$, while the \((\bm H(t), g_\gamma(t))\) pairs are numerically recovered by applying the Euler method \cite{hairer1993solving} to integrate  \eqref{eq:lq_riccati_scalar_inhomo} backward over \([0, T_\infty]\).
The limiting pair \((\bm H_\infty, g_{\gamma,\infty})\) is determined from the stationary version of \eqref{eq:ARE_homo_entropy_merged}, where \((\bm H, g_\gamma)\) is replaced by \((\bm H_\infty, g_{\gamma,\infty})\), and the time-dependent coefficients $\bm{R}$ and $\bm{Q}$ are replaced by their limiting values \(\bm{R}_\infty\) and \(\bm{Q}_\infty\).

\smallskip
\noindent\textbf{Periodic coefficients. }
We now turn to the periodic setting, i.e., let the following coefficients be $P$-periodic for some $P>0$:
\begin{equation}
\begin{aligned}
    \bm B(t+P)=\bm B(t),\,
\bm\Sigma(t+P)=\bm\Sigma(t),\,
\bm\alpha(t+P)=\bm\alpha(t), \,\\ \bm \Lambda(t+P) = \bm \Lambda(t), \, \bm Q(t+P) = \bm Q(t), \, \bm R(t+P) = \bm R(t).
\end{aligned}
\end{equation}

In this case, we seek the periodic solution $(\bm H, g_\gamma)$ of \eqref{eq:lq_pi_star_convergent_merged} -\eqref{eq:lq_riccati_scalar_inhomo}, with boundary conditions $\bm H(t+P)=\bm H(t)$, and $g_\gamma(t+P)=g_\gamma(t)$. Generally, iteration methods can be applied to solve the initial value $\bm H(0)$, and here we turn \eqref{eq:lq_riccati_scalar_inhomo} into a shooting problem \cite{bittanti1991periodic}. Then $\bm H(t)$ on $[0,P]$ can be calculated.
Once the periodic function $\bm H(t)$ is determined, $g_\gamma$ can be constructed via $g_\gamma(t)=\frac{1}{1-e^{-\beta P}}\int_0^P e^{-\beta\tau}\Big( \text{Tr}\!\big(\bm\Sigma(t+\tau) \bm\Sigma(t+\tau)^\top \bm H(t+\tau)\big) + \mathrm{Tr}(\bm\Lambda(t+\tau) \text{diag}(\bm\alpha(t+\tau))\bm H(t+\tau) \text{diag}(\bm\alpha(t+\tau))) +c_\gamma(t+\tau)\Big)\,\mathrm{d}\tau.$

We validate the method in two time-inhomogeneous settings: one with exponentially decaying coefficients and one with sinusoidally varying coefficients ($P = 10$); see Table \ref{tab:param_setting_all} for detailed parameter choices.
In both cases, we train for $N_{\mathrm{itr}}=3,000$ iterations on $[0,20]$, using step size $\delta_t=0.01$. The periodic case uses exploration intensity $\gamma = 0.05$ and the other case does not.
Figure~\ref{fig:time_inhomo_convergent} compares the learned actor and critic with the benchmark solution: the learned mean control tracks the reference feedback (panel~(a)), the value trajectory aligns with the benchmark (panel~(b)), and the state paths nearly coincide, including around jump times (panel~(c)), demonstrating good agreement despite explicit time inhomogeneity and exploratory intensities.

\begin{figure}[htpb]
    \centering
\includegraphics[width=0.95\linewidth]{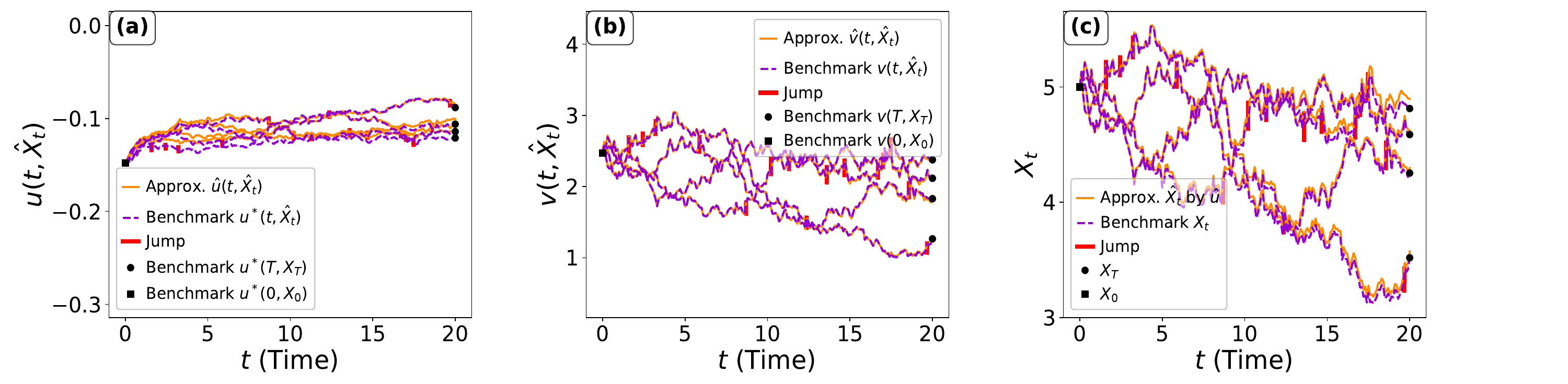}\\[2pt]
     \includegraphics[width=0.95\linewidth]{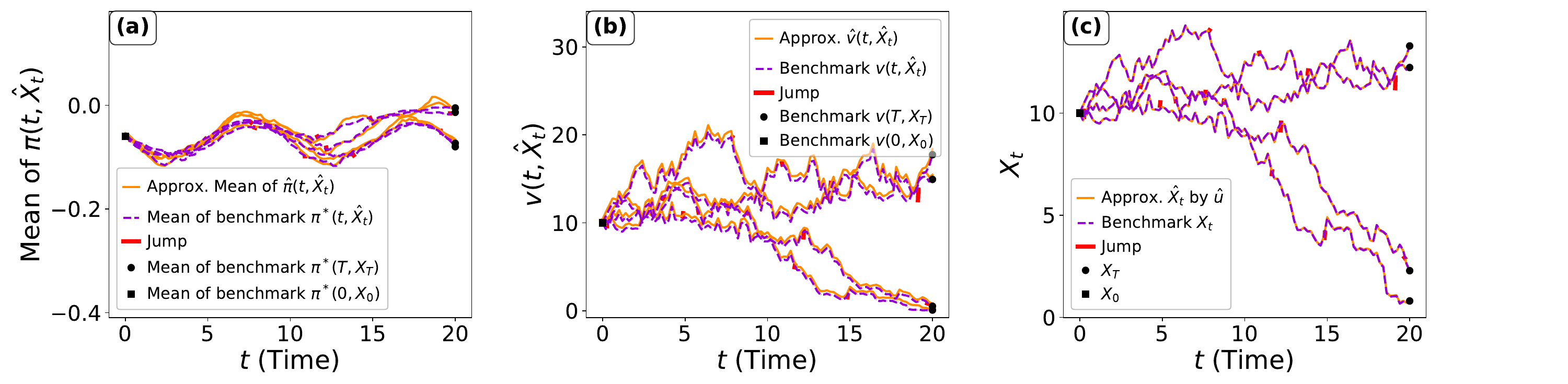}
    \caption{Predicted (a) optimal control / mean of stochastic policy, (b) value function, and (c) state trajectory for time-inhomogeneous LQ problem on horizon $T = 20$: convergent coefficients (top) and periodic coefficients (bottom).
    Parameters: $N_{\mathrm{itr}}=3{,}000$ and $\delta_t = 0.01$, with intensity $\gamma=0$ (top) and $\gamma = 0.05$ (bottom).}
    \label{fig:time_inhomo_convergent}
\end{figure}

\subsection{Merton Problem in a Jump-diffusion Market}
\label{sec:merton}

We consider Merton's portfolio optimization problem in a jump-diffusion market.
The investor chooses a strategy $(u_t)_{t\ge0}$ to allocate the fraction of the current wealth invested in the risky asset and of a risk-free asset with interest rate $r>0$. The resulting controlled wealth process $(X_t)_{t\ge0}$ satisfies
\begin{equation}
\mathrm{d}X_t
= \bigl(r + u_t(\mu - r)\bigr) X_t \, \mathrm{d}t
  + \sigma u_t X_t \, \mathrm{d}W_t
  + \alpha u_t X_t \, \mathrm{d}M_t\,.
\label{eq:merton-wealth-jump}
\end{equation}
Suppose that the investor has a reward function $f(x)=\frac{x^p}{p}$ with $0<p<1$ (i.e., CRRA utility), and seeks to maximize the expected discounted reward. Since the model parameters $(r,\mu,\sigma,\lambda,\alpha)$ and the running reward $f$ are time-homogeneous, the problem admits a stationary solution. 

\smallskip
\noindent\textbf{Standard Merton problem ($\gamma=0$).} In this case, the optimal investment fraction $u^\ast$ can be solved by $(\mu-r) +(p-1)\sigma^2 u^* +\lambda\alpha\Bigl((1+\alpha u^*)^{p-1}-1\Bigr)=0$ provided that $1+\alpha u^*>0$,
and the analytical value function $V(t,x) = V(x)=\frac{h^*}{p}x^p$, where $h^*$ satisfies $h^* \big[
p\big(r+(\mu-r)u^*\big)+\frac{1}{2}p(p-1)\sigma^2 (u^*)^2 +\lambda\bigl((1+\alpha u^*)^p-1-p\alpha u^*\bigr) -\beta\big] +1=0$.
Based on this analytical benchmark, Figure~\ref{fig:classic_merton_results} compares the learned optimal control, value function, and state trajectories with the solution. Close agreement confirms that Alg.~\ref{alg:ours} accurately recovers the classical Merton solution in this jump--diffusion setting.

\begin{figure}[htbp]
  \centering
  \begin{subfigure}{0.95\textwidth}
    \centering
\includegraphics[width=\linewidth]{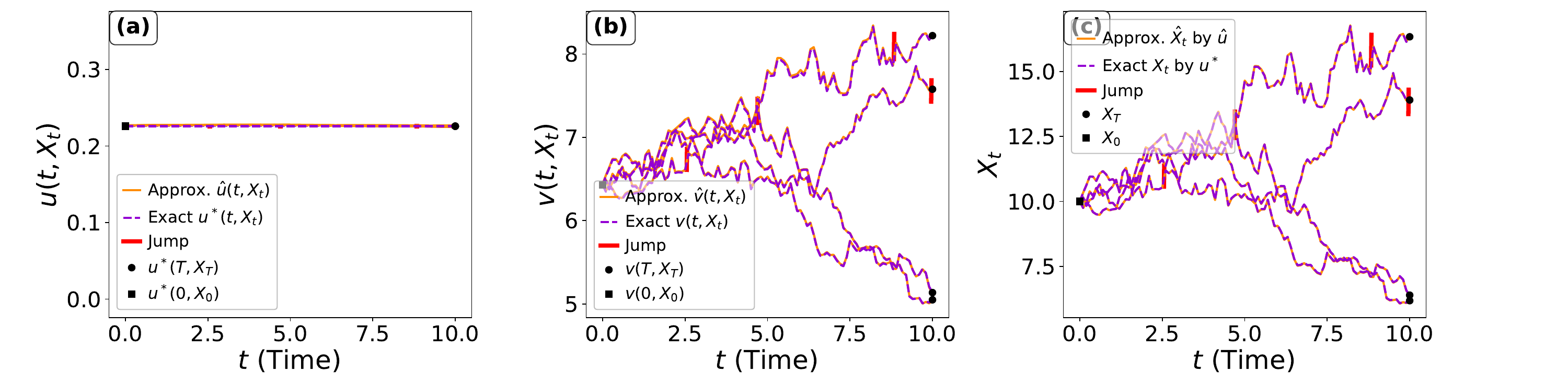}
  \end{subfigure}
  \caption{Predicted (a) optimal control, (b) value function, and (c) state trajectories for the standard Merton problem ($\gamma=0.0$) on horizon $T = 10$.
  Parameters: $\mu=0.05$, $r=0.03$, $\sigma=0.4$, $\lambda=0.2$, $\alpha=0.3$, with total iterations $N_{\text{itr}}=2{,}000$ and $\delta_t=0.01$.}
  \label{fig:classic_merton_results}
\end{figure}

\smallskip
\noindent\textbf{Entropy-regularized Merton problem ($\gamma>0$).}
According to \eqref{eq:tilde-f-def}, the running reward is given by $\tilde f(x;\pi)=\frac{x^p}{p}+\gamma \mathcal{S}(\pi)$.
The value function $V$ satisfies the entropy-regularized HJB equation in \eqref{eq:soft-HJB-opt-general}, which can be simplified as
\begin{equation}
    \label{eq:V_pinn}
    V(x)
    =
    \frac{\gamma}{\beta}
    \log \int_{\mathcal A}
    \exp\!\Big(
    \tfrac{1}{\gamma}\,
    \mathscr{H}(x,u, V, \nabla_x V,\nabla_x^2 V)
    \Big)\,\mathrm du
\end{equation}
together with the Gibbs-type optimal policy $\pi^*(u\mid x)\!\propto\! \exp\!\big(\mathscr{H}/\gamma\big)$; see supplementary materials Section~S.1.4 for a brief derivation.
In general, the optimal policy $\pi^*$ is not Gaussian, since the Hamiltonian in the present Merton setting is not quadratic in $u$, unlike in the LQ case.
As a result, under entropy regularization, the optimal policy is typically non-Gaussian, which leads to two main challenges: first, the policy distribution can no longer be accurately captured by a simple Gaussian family; second, in the absence of an explicit expression for the optimal policy, the value function $V$ generally does not admit a closed-form characterization.
Fortunately, $V$ can be computed numerically to high accuracy via a physics-informed neural network (PINN) solver \cite{raissi2019physics} applied to \eqref{eq:V_pinn}. The resulting numerical approximation of the value function can then be used to recover the corresponding optimal policy, thereby providing a benchmark solution. Full implementation details of the PINN solver are deferred to Appendix~\ref{subsec:ac_arch}.

With such a benchmark, our parameterization of \(\pi_\theta\) using conditional normalizing flows, introduced in Section~\ref{sec:numerical_implementation}, provides a flexible framework for representing general non-Gaussian distributions. Figure~\ref{fig:exploratory_merton_results} compares the conditional policy distributions at different time points, the value function, and the state trajectories. The parameter settings are reported in Table~\ref{tab:param_setting_all}, while the trajectory construction is described in Appendix~\ref{subsec:ac_arch}. The close agreement with the PINN benchmark demonstrates the effectiveness of our algorithm and of the flow-based policy parameterization.

\begin{figure}[htbp]
  \centering
  \begin{subfigure}{0.85\textwidth}
    \centering
    \includegraphics[width=\linewidth]{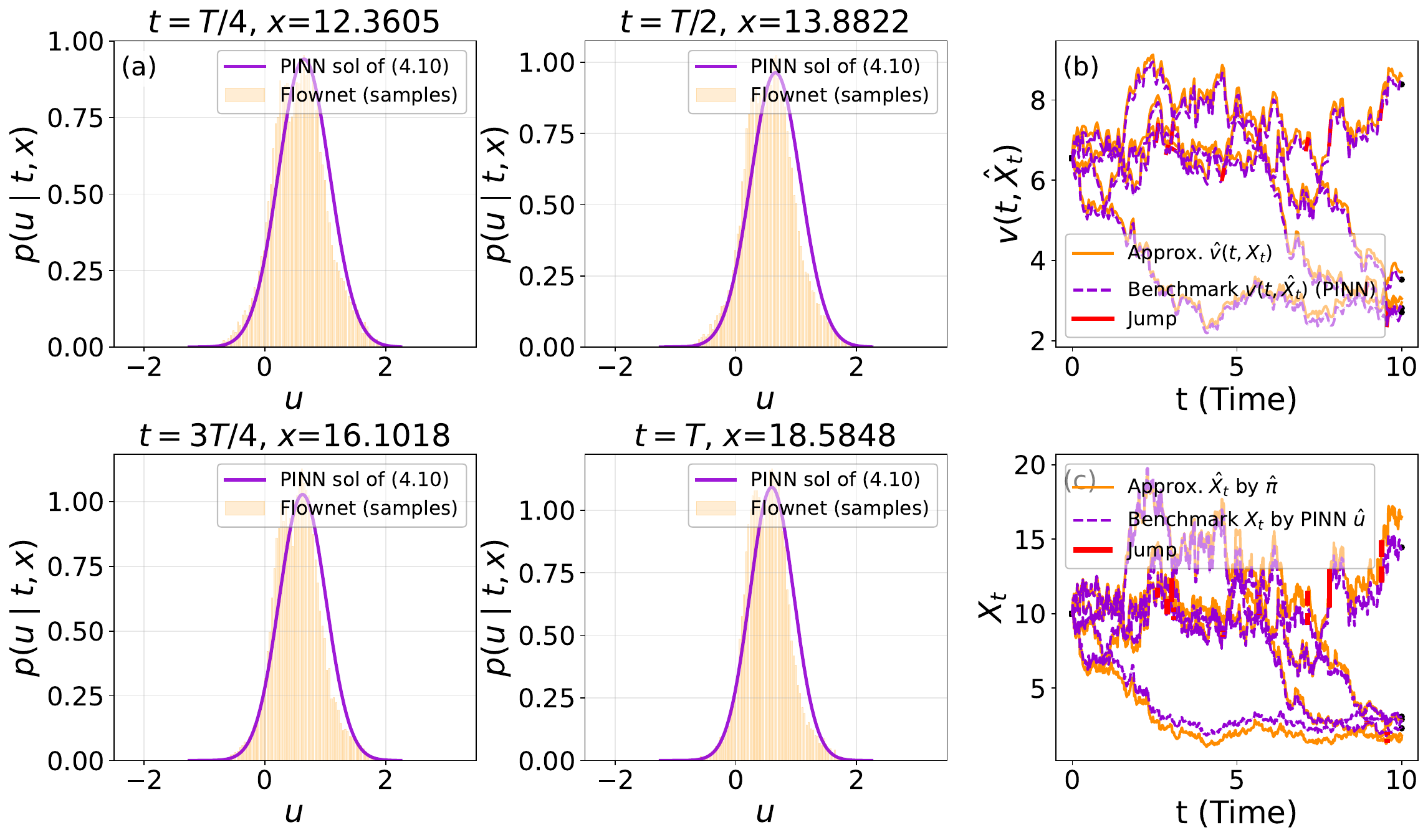}
  \end{subfigure}
  \caption{Plots of (a) densities $p(u|\, t, \bm x)$ of the stochastic policy $\pi(u | t, \bm x)$ at $t=T/4$, $T/2,\, 3T/4, \, T$ and some $\bm x$,  (b) value function, and (c) state trajectories for the entropy-regularized Merton problem ($\gamma=0.05$) on horizon $T=10$.
  The total iterations is $N_{\text{itr}}=2{,}000$ and $\delta_t=0.05$.}
  \label{fig:exploratory_merton_results}
\end{figure}

\subsection{Multi-Agent Portfolio Game in Jump-diffusion Market}
\label{sec:multi-agent}

Our final example considers a game-theoretic extension of the Merton problem presented in Section \ref{sec:merton}, where each agent's reward depends on performance relative to the population average. This example serves two purposes: it demonstrates that the proposed algorithm scales well to high-dimensional systems and that it remains effective in the presence of strategic interactions.

We consider $n$ agents, indexed by $i\in\{1,\dots,n\}$, each choosing a control process $u_i$ to maximize their own expected discounted reward. The wealth process of agent $i$ evolves as
\begin{equation}
dX_t^i = u_i\bigl(b_i\,\mathrm{d}t+ \eta_i\,\mathrm{d} W_t^i + \sigma_i\,\mathrm{d}W_t^0
+ \alpha_i\,\mathrm{d}M_t ^i + \xi_i\,\mathrm{d}M_t^0\bigr),
\label{eq:SDE-X-i}
\end{equation}
where \(W^i\) is the idiosyncratic Brownian motion of agent \(i\), \(W^0\) is the common Brownian motion shared by all agents, \(M^i\) is the compensated Poisson jump process specific to agent \(i\), and \(M^0\) represents common jump shocks.
Fixing the strategies $\bm u_{-i}$ of all other agents, agent $i$ chooses $u_i$ to maximize
\begin{equation}\label{eq:Jgame}
\small
\begin{aligned}
J^i(t,x,y;u_i, \bm u_{-i})
:=
\mathbb E\Big[
\int_t^\infty e^{-\beta(s-t)}
\, f_i\bigl(X_s^i,Y_s^i\bigr)\,\mathrm{d}s
\;\big|\; X_t^i=x,\;Y_t^i=y
\Big], 
 \end{aligned}
\end{equation}
where $Y_t^i := \frac{1}{n}\sum_{j\neq i} X_t^j$ represents the average wealth of the other agents, and reward $f_i(x, y) = -\exp\Big(
 -\frac{1}{\varrho_i}
 \Big(\big(1-\tfrac{\varpi_i}{n}\big)x - \varpi_i y\Big)
 \Big)$ measures relative performance, with risk-tolerance parameter $\varrho_i>0$ and competition weight $\varpi_i\in\mathbb R$.

The goal of this multi-agent game is to find a Nash equilibrium, namely a collection of controls $\bm u^* = (u_1^*,\dots,u_n^*)$ such that no agent can improve its own objective by deviating unilaterally while the others keep their strategies fixed. In other words, for every agent $i$, the control $u_i^*$ is an optimal response to $\bm u_{-i}^*$.
To characterize such an equilibrium, we first solve the best-response problem of a representative agent and then impose consistency across all agents. Fixing the strategies $u_j$ of all agents $j\neq i$, we define the best-response value function of agent $i$ by
\begin{equation}\label{def:Vi-game}
V^i(t,x,y)
:= \sup_{(u_i(s))_{s\ge t}} J^i(t,x,y;u_i,\bm u_{-i}).
\end{equation}

For this portfolio game, the analytical benchmark is characterized by a coupled first-order system: the equilibrium investment strategy $\bm u^\ast = (u_1^\ast,\ldots,u_n^\ast)$ satisfies $\Psi_i'(u_i^\ast)=0, i=1,2,\ldots,n,$ where $\Psi_i(u)
=
-\chi_i b_i u
+ \frac12\chi_i^2(\eta_i^2+\sigma_i^2)u^2
- \chi_i\rho_i\sigma_i\widehat{u\sigma}\,u
+ \lambda_i\bigl(e^{-\chi_i\alpha_i u}-1+\chi_i\alpha_i u\bigr)
+ \lambda_0\Bigl(
e^{-\chi_i\xi_i u + \rho_i\widehat{u\xi}} + \chi_i\xi_i u
\Bigr)$, 
with $\widehat{u\sigma}:= \frac1n\sum_{j\neq i} u_j\sigma_j,
\, \widehat{u\xi}:= \frac1n\sum_{j\neq i} u_j\xi_j,
\, \chi_i := \frac{1-\varpi_i/n}{\varrho_i},
\, \rho_i := \frac{\varpi_i}{\varrho_i}.$
Correspondingly, the value function of agent $i$ takes the form
\begin{equation}
\small
V^i(x,y)
= -\frac{1}{\beta - \Lambda_i^*}
\exp\big(
-\frac{1}{\varrho_i}
\big(\big(1-\tfrac{\varpi_i}{n}\big)x - \varpi_i y\big)
\big),
\label{eq:optimal_V_multi-CARA}
\end{equation}
provided that $\beta > \Lambda_i^*$, where $\Lambda_i^* := \Psi_i(u_i^*) + C_i,$ and 
$C_i = \rho_i \,\widehat{ub} + \frac{1}{2}\,\rho_i^2 \big(\frac{1}{n^2}\sum_{j\neq i}(u_j \eta_j)^2 \\ + \big(\frac{1}{n}\sum_{j\neq i} u_j \sigma_j\big)^2
  \big) + \sum_{j\neq i}\lambda_j\big(
      \exp\big(\frac{\varpi_i}{\varrho_i}\,\frac{u_j \alpha_j}{n}\big) - 1
      - \frac{\varpi_i}{\varrho_i}\,\frac{u_j \alpha_j}{n}
  \big)
- \lambda_0 - \lambda_0\,\rho_i\,\widehat{u\xi}$ is independent of $u_i$. Here $\widehat{ub} := \frac1n\sum_{j\neq i} u_j b_j.$
The proof of the above characterization can be found in supplementary materials Section~\suppref{appen:proof_multi-agent_game}; we use it as the analytical baseline when evaluating our numerical method.

\begin{figure}[tb]
  \centering
  \begin{subfigure}{0.95\textwidth}
    \centering
    \includegraphics[width=\linewidth]{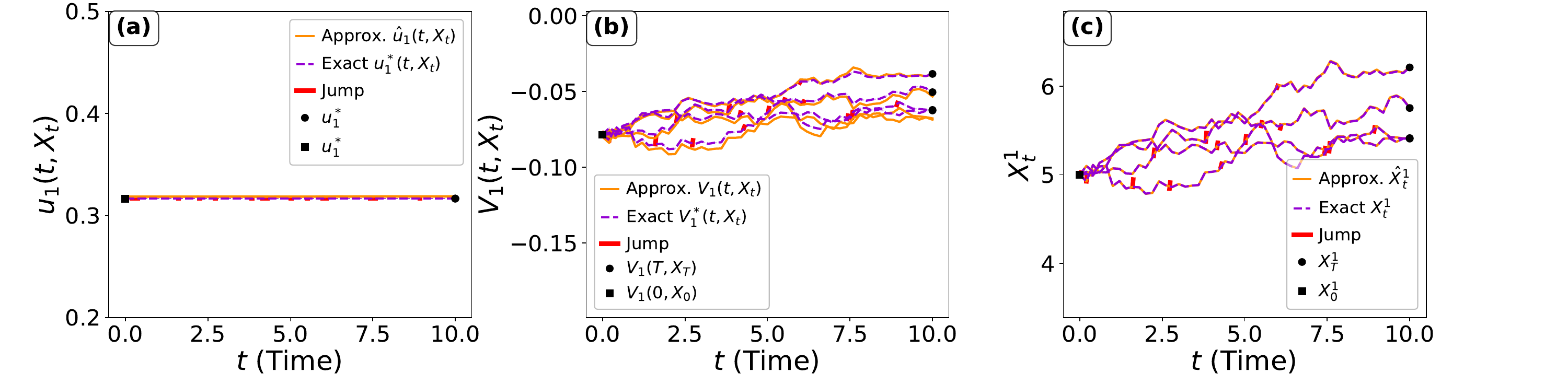}
  \end{subfigure}
  \caption{Comparison of (a) optimal control, (b) value function and (c) state trajectories vs. benchmarks for agent $1$ in the Merton portfolio game. Results are shown after $N_{\text{itr}} = 1,000$ training iterations with $\delta_t=0.02$.}
  \label{fig:multi-agent-D25}
\end{figure}

Figures~\ref{fig:multi-agent-D25}--\ref{fig:multi-agent-D25-loss} illustrate the numerical results with $n=25$ agents.
During training, each agent has its own actor and critic networks, and all agents can be trained in parallel. At every iteration each agent updates its parameters against the other agents' policies from the previous iteration, aiming to converge to the Nash equilibrium.
We consider a heterogeneous setting where agent 1 faces a different market and preference parameters, while the rest are homogeneous. The full parameter settings are reported in Appendix~\ref{sec:numericalsupplements} Table~\ref{tab:param_setting_all}. All plots are generated after $N_{\text{itr}} = 1,000$ training iterations with time step  $\delta_t=0.02$. Figure \ref{fig:multi-agent-D25} depicts the predicted control, value function, and state trajectories for agent 1. Figure \ref{fig:multi-agent-D25-loss} displays the RMSEs of the control and the value function for each agent, together with the averaged training loss of all agents. 
Overall, the close alignment with the benchmarks, together with the small RMSEs and stable training loss, indicates that our method achieves promising and stable performance even in high dimensions and in the presence of strategic interactions.

\begin{figure}[htbp]
  \centering
  \begin{subfigure}{0.95\textwidth}
    \centering
\includegraphics[width=\linewidth]{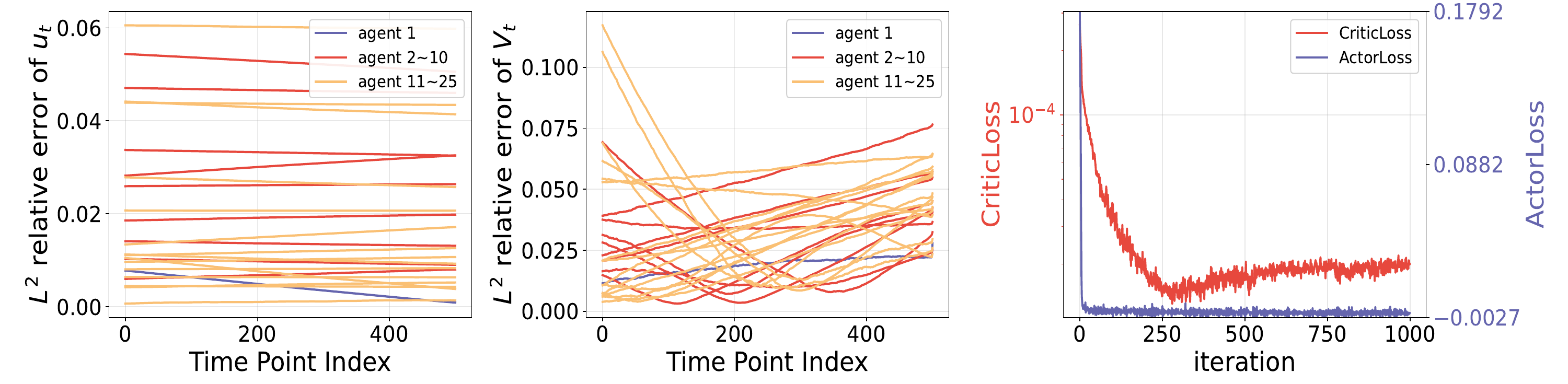}
  \end{subfigure}
  \caption{MSEs of the learned controls and value functions for each agent ($n = 25$), together with the training loss averaged over all agents. The full parameter setting is reported in Table~\ref{tab:param_setting_all}.}
  \label{fig:multi-agent-D25-loss}
\end{figure}

Table \ref{tab:runtime_nd} reports the runtime and RMSEs (Eqs. \eqref{eq:error_X_V_rel}--\eqref{eq:error_u_rel}) for different choices of time step $\delta_t$ and number of agents $n$. The runtime increases approximately linearly with $n$, indicating favorable scalability with respect to the problem dimension. 
The errors remain comparable across these configurations, indicating stable performance as the problem size grows.

\begin{table}[htpb]
\centering
\caption{Runtime and last-iteration relative errors for different (number of time points,\, number of agents) pairs $=(K, n)$ of our algorithm for multi-agent game.}
\label{tab:runtime_nd}
\footnotesize
\renewcommand{\arraystretch}{0.9}
\begin{tabular}{llcccc}
\toprule
 &  & $n{=}2$ & $n{=}5$ & $n{=}10$ & $n{=}25$ \\
\midrule
\multirow{3}{*}{$K=100$}
& Runtime (min) & 18.98 & 40.67 & 76.03 & 205.76 \\
& $\mathcal{E}_{V}$ & 0.1853 & 0.1866 & 0.1490 & 0.1767 \\
& $\mathcal{E}_{\text{u}}$ & 0.0300 & 0.0201 & 0.0330 & 0.0318 \\
\midrule
\multirow{3}{*}{$K=500$}
& Runtime (min) & 94.72 & 202.19 & 400.77 & 1050.64 \\
& $\mathcal{E}_{V}$ & 0.0617 & 0.0355 & 0.0447 & 0.0476 \\
& $\mathcal{E}_{\text{u}}$ & 0.0081 & 0.00137 & 0.0176 & 0.0261 \\
\bottomrule
\end{tabular}%
\end{table}

\section{Conclusions and Discussions}\label{sec:conclusion}

This paper develops a reinforcement learning framework for infinite-horizon time-inhomogeneous stochastic control problem subject to jump-diffusion and entropy regularization. We introduce a continuous-time little $q$-function, define an appropriate time-dependent occupation measure, and establish its structural properties. This representation leads to a general policy gradient formula, and we design an actor-critic algorithm tailored to general time-inhomogeneous jump-diffusion dynamics and non-Gaussian stochastic policies via conditional normalizing flow. We also derive explicit solutions for the value function and the optimal (stochastic) policy for several canonical specifications, including LQ control, the Merton portfolio problem, and multi-agent portfolio game with CARA utilities in jump-diffusion markets. These closed-form characterizations provide ground-truth benchmarks for evaluating RMSEs of the proposed algorithm.
Our method  is validated on a suite of low- and high-dimensional experiments, including settings with jump components and time-dependent coefficients. Across multiple evaluation metrics, the learned policies closely track the analytic solutions when available, and our algorithm exhibits strong performance broadly.

There are several natural directions for future work. Building on the proposed actor-critic framework, it would be interesting to study how alternative entropy regularizers affect the control problem, for example the Tsallis entropy considered in \cite{bo2024continuous}. However, its structure typically precludes closed-form optimal stochastic policies, making benchmarking more challenging. On the theoretical side, extending the little $q$-framework and occupation measure to partially observed models or mean-field interaction structures would further bridge continuous-time RL and modern stochastic control with jumps. On the algorithmic side, scaling our approach to very high-dimensional problems and integrating it with large language models or agent-based architectures are promising directions; recent progress on agent-based methods and DFA-type accelerators suggests substantial potential for speeding up RL in complex control environments \cite{alsadat2024multi,alsadat2025using}.

\bibliographystyle{plain}
\bibliography{ref}

\appendix
\section{More Numerical Details}\label{sec:numericalsupplements}

\subsection{Neural Network (NN) Architectures and Experimental Details}
\label{subsec:ac_arch}

\smallskip
\noindent\textbf{Critic network.}
We parameterize the critic $V_\psi:(t,\bm x)\mapsto\mathbb{R}$ by a ResNet \cite{he2016deep} of depth $3$,  with input dimension $d+1$, hidden width $d+10$, \texttt{tanh} activations, and a scalar linear readout. All critic networks are optimized with Adam (learning rate~$10^{-3}$) and a fixed random seed ($2025$).  For the LQ cases and the multi-player game we decay the learning rate with a multi-step schedule; for the Merton problem with flow layers we use a \texttt{CosineAnnealingWarmUp} schedule.

\noindent\textbf{Actor network.}
We parameterize the actor $\pi_\theta(\cdot\mid t,\bm x)$ as a conditional normalizing flow. Given $(t,\bm x)$, a ResNet with \texttt{tanh} activations, depth~$3$, and hidden width~$d+10$ outputs the parameters of the Gaussian base distribution in \eqref{eq:policy_parameterized_impl}. 
For the standard case with $\gamma=0$, we use a fixed standard deviation $\mathrm{Std}=0.1$ instead of a learnable $\mathrm{Std}_\theta$. This follows \cite{montenegro2024learning,wang2020reinforcement}, where fixing the exploration scale is found to improve stability when stochastic policies are used to learn an underlying deterministic optimal control, while still being sufficient for convergence. This applies to all LQ cases and the multi-player Merton game.
The actor is optimized with Adam using learning rate $10^{-3}$ for time-homogeneous problems and $5\times 10^{-4}$ for time-inhomogeneous problems, together with a multi-step learning-rate schedule.

For the entropy-regularized Merton problem, samples from the Gaussian base distribution are further transformed by an invertible flow map~$F_\theta$. 
Specifically, we employ a conditional monotone rational-quadratic spline coupling~\cite{durkan2019neural} on~$\mathbb{R}^m$ for $F_\theta$, where each coordinate transformation is represented by a piecewise rational-quadratic spline. 
We use $K_{\mathrm{bin}}=6$ bins with identity tails outside $[-B_{\mathrm{flow}},B_{\mathrm{flow}}]=[-2.5,2.5]$. 
The spline parameters are produced by a ResNet with \texttt{tanh} activations, depth~$2$, and hidden width~$32$. 
To improve stability in early training, we initialize the flow close to the identity map and freeze its parameters for the first $30$ actor updates. 
After this warm-up stage, the flow is applied directly, i.e., $\bm z_K = F_\theta(\bm z_0)$. 
Bounded actions are enforced through a temperature-controlled sigmoid squashing function $\bm u=S(\bm z_K;\tau)$. 
We anneal $\tau$ from $2.0$ to $1.0$ over the first $30$ squash-delay steps so that the squashing is milder during warm-up and actions are less likely to saturate near the boundaries.

When the flow $F_\theta$ is enabled, the actor's base network and the spline flow layers are trained using an Adam optimizer with two parameter groups. 
The base network uses $6\times 10^{-5}$, while the flow layers use learning rate $1\times 10^{-5}$. 
We choose a smaller learning rate for the flow because the spline--squash composition is more sensitive to parameter updates, and a lower step size improves training stability. 
Both parameter groups are scheduled jointly by a single \texttt{CosineAnnealingWarmUp} scheduler.

\noindent\textbf{PINN approach for Merton's problem.}
We compute a reference solution by solving simplified HJB equation \eqref{eq:V_pinn} on uniform grids $x_{\mathrm{grid}}\subset [x_{\min},x_{\max}]$ and $u_{\mathrm{grid}}\subset [u_{\min},u_{\max}]$, with $n_x=500$ grid points for $x$ and $n_u=400$ grid points for $u$.
The value function in PINN approach is parameterized by a feedforward NN $V_\phi(x)$ with $5$ hidden layers, each of width $256$, and \texttt{tanh} activations. All first- and second-order derivatives of $V_\phi$ entering the Hamiltonian $\mathscr{H}(x,u, V_{\phi},\nabla_x V_\phi,\nabla_x^2 V_\phi)$ are computed via automatic differentiation.
To approximate the integral in \eqref{eq:V_pinn}, we use the trapezoidal quadrature on $u_{\mathrm{grid}}$, denoted by $\mathcal{Q}_u[\cdot]$. For each $(x_i,u_j)$, let $H_{ij}=\mathscr{H}\!\bigl(x_i,u_j, V_\phi(x_i),\nabla_x V_\phi(x_i),\nabla_x^2 V_\phi(x_i)\bigr)$. Then the right-hand side of \eqref{eq:V_pinn} is approximated by $V_{\mathrm{rhs}}(x_i)=\frac{\gamma}{\beta}\log\!\big(\mathcal{Q}_u[\exp(H_{ij}/\gamma)]\big)$. In practice, for numerical stability, we evaluate it through a log-sum-exp implementation of the quadrature.
The network parameters $\phi$ are trained by minimizing the mean squared residual $\mathcal{L}_{\mathrm{PINN}}(\phi) = \frac{1}{n_x}\sum_{i=1}^{n_x} \bigl(
V_\phi(x_i)-V_{\mathrm{rhs}}(x_i)
\bigr)^2,$ using Adam with learning rate $5\times 10^{-4}$ for $2,000$ iterations.

After obtaining the PINN solution of the value function, we recover the optimal policy on the grid $u_{\mathrm{grid}}$ as the Gibbs distribution induced by the same quadrature normalization, $\pi_{ij} = \frac{w_j \exp(H_{ij}/\gamma)}{\sum_{k=1}^{n_u} w_k \exp(H_{ik}/\gamma)},$ where $\{w_j\}_{j=1}^{n_u}$ are the trapezoidal weights. 

\subsection{Parameter Setting}\label{appen:parameter}

Table \ref{tab:param_setting_all} records parameters in our experiments.

\begin{table}[p]
\centering
\caption{Parameter settings in numerical experiments.}
\label{tab:param_setting_all}
\renewcommand{\arraystretch}{1.22}
\setlength{\tabcolsep}{5pt}

\rotatebox{90}{
\resizebox{0.9\textheight}{0.45\linewidth}{%
\begin{tabular}{@{}l|cc|c|cc|c@{}}
\toprule
& \multicolumn{2}{c|}{\textbf{Time-inhomogeneous LQ}} 
& \textbf{LQ (Homo.)} 
& \multicolumn{2}{c|}{\textbf{Merton}}
& \textbf{Multi-Agent CARA} \\
\cmidrule(lr){2-3}\cmidrule(lr){5-6}
& \textbf{Convergent} & \textbf{Periodic} & & \textbf{Standard} & \textbf{Entropy} & \\
\midrule

Dim / agent
& \multicolumn{2}{c|}{$d=1$}
& $d = 1, 5, 20, 50$
& \multicolumn{2}{c|}{$d = 1$}
& $n = 25$ \\[2pt]
\midrule

Drift coeff.
& \makecell[l]{$B_t u_t=b(t)\,u_t,$\\
$b(t)=b_\infty+(b_0-b_\infty)e^{-\kappa_b t}$\\
$b_0=0.6,\ b_\infty=0.5,\ \kappa_b=1.0$}
& \makecell[l]{$B_t u_t=b(t)\,u_t,$\\
$b(t)=b_{\mathrm{bar}}+b_{\mathrm{amp}}\sin\!\big(\tfrac{2\pi t}{P_b}+\phi_b\big)$\\
$b_{\mathrm{bar}}=0.12,\ b_{\mathrm{amp}}=0.06,$\\
$P_b=10.0,\ \phi_b=0.0$}
& $B \mathbf u_t = 0.5\,\mathbf I\,\mathbf u_t$
& \multicolumn{2}{c|}{\makecell[l]{$(r+u_t(\mu-r))X_t$\\
$\mu = 0.05$ \quad $\mu=0.1$\\
$r = 0.03$ \quad $r=0.05$}}
& \makecell[l]{$b_i=
\begin{cases}
0.05, & i=1,\\
0.02, & 2\le i\le n
\end{cases}$} \\[2pt]
\midrule

Diffusion coeff.
& \makecell[l]{$\Sigma_t=\sigma(t),$\\
$\sigma(t)=\sigma_\infty+(\sigma_0-\sigma_\infty)e^{-\kappa_\sigma t}$\\
$\sigma_0=0.3,\ \sigma_\infty=0.2,\ \kappa_\sigma=1.0$}
& \makecell[l]{$\Sigma_t=\sigma(t),$\\
$\sigma(t)=\sigma_{\mathrm{bar}}+\sigma_{\mathrm{amp}}\sin\!\big(\tfrac{2\pi t}{P_\sigma}+\phi_\sigma\big)$\\
$\sigma_{\mathrm{bar}}=0.2,\ \sigma_{\mathrm{amp}}=0.1,$\\
$P_\sigma=10.0,\ \phi_\sigma=0.0$}
& $\bm\Sigma = 0.3\,\mathbf I$
& \multicolumn{2}{c|}{$\sigma = 0.4$}
& \makecell[l]{$\eta_i=
\begin{cases}
0.08, & i=1,\\
0.05, & 2\le i\le n
\end{cases}$\\
$\sigma_i=
\begin{cases}
0.5, & i=1,\\
0.4, & 2\le i\le n
\end{cases}$} \\[2pt]
\midrule

Jump intensity
& \multicolumn{2}{c|}{$\lambda_t\equiv \lambda,\ \lambda=0.2$}
& \makecell[l]{$\bm\lambda=\mathrm{linspace}(0.2,0.3,D)$}
& $\lambda = 0.2$ & $\lambda = 0.3$
& \makecell[l]{$\lambda_i \equiv 0.2$\\
$\lambda_0 = 0.25$} \\[2pt]
\midrule

Jump sizes / coeff.
& \makecell[l]{$\alpha_t=\alpha(t),$\\
$\alpha(t)=\alpha_\infty+(\alpha_0-\alpha_\infty)e^{-\kappa_\alpha t}$\\
$\alpha_0=0.3,\ \alpha_\infty=0.2,\ \kappa_\alpha=1.0$}
& \makecell[l]{$\alpha_t=\alpha(t),$\\
$\alpha(t)=\alpha_{\mathrm{bar}}+\alpha_{\mathrm{amp}}\sin\!\big(\tfrac{2\pi t}{P_\alpha}+\phi_\alpha\big)$\\
$\alpha_{\mathrm{bar}}=0.2,\ \alpha_{\mathrm{amp}}=0.1,$\\
$P_\alpha=10.0,\ \phi_\alpha=0.0$}
& \makecell[l]{$\bm\alpha=\mathrm{linspace}(0.3,0.2,D)$}
& $\alpha = 0.3$ & $\alpha = 0.1$
& \makecell[l]{$\alpha_i \equiv 0.2$\\
$\xi_i \equiv 0.2$} \\[2pt]
\midrule

Discount factor
& \multicolumn{2}{c|}{$\beta=1.0$}
& $\beta = 1.0$
& \multicolumn{2}{c|}{$\beta=1.0$}
& $\beta = 1.0$ \\[2pt]
\midrule

Exploration intensity $\gamma$
& $\gamma = 0.0$
& $\gamma = 0.05$
& $\gamma = 0 \text{ or } 0.05$
& $\gamma = 0$
& $\gamma = 0.05$
& $\gamma = 0$ \\[2pt]
\midrule

Learning rate (actor)
& \multicolumn{2}{c|}{$5\times 10^{-4}$}
& $10^{-3}$
& $10^{-3}$
& $6\times 10^{-5}$
& $10^{-3}$ \\[2pt]
\midrule

Learning rate (critic)
& \multicolumn{2}{c|}{$10^{-3}$}
& $10^{-3}$
& \multicolumn{2}{c|}{$10^{-3}$}
& $10^{-3}$ \\[2pt]
\midrule

Iterations
& \multicolumn{2}{c|}{$N_{\mathrm{itr}}=3000$}
& $1{,}000$
& $2{,}000$
& $2{,}000$
& $1{,}000$ \\[2pt]
\midrule

Time step size / $\delta_t$
& \multicolumn{2}{c|}{$T=20,\ K=2000,\ \delta_t=T/K=0.01$}
& $\delta_t=0.01$
& $\delta_t=0.01$
& $\delta_t=0.05$
& $\delta_t=0.02$ \\[2pt]
\midrule

Minibatch size
& \multicolumn{2}{c|}{$L=100$}
& $L = 100$
& $L = 500$
& $L = 200$
& $L = 100$ \\[2pt]
\midrule

Note
& \multicolumn{2}{c|}{\makecell[c]{$\bm R = 2 \bm I_d,\ \bm Q = 0.1 \bm I_d$\\
$K_{\mathrm{actor}} = 15,\ K_{\mathrm{critic}} = 5$}}
& \makecell[l]{$\bm R = 5 \bm I_d,\ \bm Q = 0.5 \bm I_d$\\
$K_{\mathrm{actor}} = 20,\ K_{\mathrm{critic}} = 5$}
& \multicolumn{2}{c|}{\makecell[l]{$p = 0.5$\\
$K_{\mathrm{actor}} = 20,\ K_{\mathrm{critic}} = 5$}}
& \makecell[l]{$\varpi_i \equiv 0.2$\\
$\varrho_i=
\begin{cases}
1.5, & i=1,\\
2.0, & 2\le i\le n
\end{cases}$\\
$K_{\mathrm{actor}} = 30,\ K_{\mathrm{critic}} = 10$} \\
\midrule

Riccati reference solve
& \multicolumn{2}{c|}{$N_{\mathrm{riccati}}=3000,\ \mathrm{d}t_{\mathrm{ODE}}=10^{-3},\ T_{\mathrm{trunc}}=3T$}
& \multicolumn{1}{c|}{--}
& \multicolumn{2}{c|}{--}
& -- \\
\bottomrule
\end{tabular}%
}
}
\end{table}

\newpage

\ifdefined\arxivcombined
  \clearpage
  \setcounter{section}{0}
  \setcounter{theorem}{0}
  \setcounter{figure}{0}
  \setcounter{table}{0}
  \setcounter{algorithm}{0}
  \renewcommand{\theHsection}{supp.\arabic{section}}
  \renewcommand{\theHtheorem}{supp.\arabic{section}.\arabic{theorem}}
  \renewcommand{\theHfigure}{supp.\arabic{figure}}
  \renewcommand{\theHtable}{supp.\arabic{table}}
  \renewcommand{\theHalgorithm}{supp.\arabic{section}.\arabic{algorithm}}
  \renewcommand{\thefigure}{S\arabic{figure}}
  \renewcommand{\thetable}{S\arabic{table}}
  \begin{center}
    {\large\bfseries Supplementary Materials\par}
  \end{center}
  \input{ex_supplement_body}
\fi

\end{document}

%% file: ex_shared.tex
\usepackage{amsfonts}
\usepackage{graphicx}
\usepackage{algpseudocode}
\DeclareGraphicsExtensions{.pdf,.png,.jpg,.jpeg}

\usepackage{mathtools}
\usepackage{bm}
\usepackage{mathrsfs}
\usepackage{algorithm}
\usepackage{subcaption}
\usepackage{booktabs}
\usepackage{multirow}
\usepackage{makecell}
\usepackage{pdfpages}
\usepackage{pdflscape}
\usepackage[numbers,sort&compress]{natbib}
\usepackage{tikz}
\usetikzlibrary{arrows.meta, positioning, shapes.geometric, calc}
\usepackage{wrapfig}
\usepackage{rotating}
\usepackage{comment}

\usepackage[capitalize,nameinlink]{cleveref}

\crefname{hypothesis}{Hypothesis}{Hypotheses}
\crefname{fact}{Fact}{Facts}
\numberwithin{algorithm}{section}
\title{A Continuous-Time Actor-Critic Framework for Jump-Diffusion Controls with Normalizing Flows\thanks{Y.Z.\ and L.G.\ were partially supported by the National Key R\&D Program of China (grant 2021YFA0719200). R.H. was partially supported by the ONR grant N00014-24-1-2432, the Simons Foundation (MP-TSM-00002783), and the NSF grant DMS-2420988.}}

\author{Liya Guo\thanks{Yau Mathematical Sciences Center, Tsinghua University, Beijing 100084, China; and Department of Mathematics, Tsinghua University, Beijing 100084, China
  (\email{gly22@mails.tsinghua.edu.cn}).}
\and Ruimeng Hu\thanks{Department of Mathematics, and Department of Statistics and Applied Probability,
University of California, Santa Barbara, CA 93106-3080, USA
  (\email{rhu@ucsb.edu}).}
\and Xu Yang\thanks{Department of Mathematics, University of California, Santa Barbara, CA 93106-3080, USA
  (\email{xuyang@math.ucsb.edu}).}
\and Yi Zhu\thanks{Yau Mathematical Sciences Center, Tsinghua University, Beijing 100084, China; and
Yanqi Lake Beijing Institute of Mathematical Sciences and Applications, Beijing 101408, China
  (\email{yizhu@tsinghua.edu.cn}).}}

\usepackage{amsopn}

\newcommand{\R}{\mathbb{R}}

\newcommand{\MRshift}{-0.65em}

\colorlet{RED}{red}

\makeatletter
\@ifundefined{myeqref}{\newcommand{\myeqref}[1]{%
  \@ifundefined{myeq@#1}%
    {\textup{(??)}\PackageWarning{myeqref}{Unknown label: #1}}%
    {\textup{(\csname myeq@#1\endcsname)}}%
}}{}
\@ifundefined{myref}{\newcommand{\myref}[1]{%
  \@ifundefined{myrf@#1}%
    {??\PackageWarning{myref}{Unknown label: #1}}%
    {\csname myrf@#1\endcsname}%
}}{}
\@ifundefined{suppeqref}{\newcommand{\suppeqref}[1]{%
  \@ifundefined{suppeq@#1}%
    {\textup{(??)}\PackageWarning{suppeqref}{Unknown supplement label: #1}}%
    {\textup{(\csname suppeq@#1\endcsname)}}%
}}{}
\@ifundefined{suppref}{\newcommand{\suppref}[1]{%
  \@ifundefined{supprf@#1}%
    {??\PackageWarning{suppref}{Unknown supplement label: #1}}%
    {\csname supprf@#1\endcsname}%
}}{}
\@namedef{myeq@eq:q_function}{3.8}
\@namedef{myeq@eq:Qderive}{3.7}
\@namedef{myeq@eq:Q_delta}{3.6}
\@namedef{myeq@eq:Jgamma-shift}{2.5}
\@namedef{myeq@eq:explor_SDE_PO}{2.2}
\@namedef{myeq@eq:mart_inc_J_corrected}{3.3}
\@namedef{myeq@eq:discounted-dynkin-shift-final}{3.10}
\@namedef{myeq@eq:perf_diff}{3.11}
\@namedef{myeq@eq:policy-eval-PIDE}{2.6}
\@namedef{myeq@eq:gae_def_J}{3.14}
\@namedef{myeq@eq:gae-first-order_J}{3.15}
\@namedef{myeq@eq:V_pinn}{4.10}
\@namedef{myrf@sec:q_func}{3.2.1}
\@namedef{myrf@lemma:gen_property_shift-final}{3.3}
\@namedef{myrf@thm:PG}{3.4}
\@namedef{myrf@lem:occupation-identity-inhom-shift}{3.2}
\@namedef{myrf@lem:gae-approximation}{3.5}
\@namedef{supprf@appen:derivation_q}{S.1.1}
\@namedef{supprf@appen:proof_lemma2}{S.1.2}
\@namedef{supprf@appen:proof_of_thm2}{S.1.2}
\@namedef{supprf@appen:derivation_gae}{S.1.3}
\@namedef{supprf@sec:benchmark}{S.2}
\@namedef{supprf@sec:lq_closed_form}{S.2.1}
\@namedef{supprf@appen:proof_multi-agent_game}{S.2.2}
\@namedef{supprf@sec:discretization}{S.3.1}
\@namedef{supprf@sec:mart_disc}{S.3.2}
\@namedef{suppeq@eq:Xdiscrete}{S.3.1}
\makeatother

\ifdefined\arxivcombined
  \renewcommand{\myeqref}[1]{\eqref{#1}}
  \renewcommand{\myref}[1]{\ref{#1}}
  \renewcommand{\suppeqref}[1]{\eqref{#1}}
  \renewcommand{\suppref}[1]{\ref{#1}}
\fi

%% file: ex_supplement_body.tex
\renewcommand{\thesection}{S.\arabic{section}}
\renewcommand{\thesubsection}{S.\arabic{section}.\arabic{subsection}}
\renewcommand{\theequation}{S.\arabic{section}.\arabic{equation}}
\renewcommand{\thetheorem}{S.\arabic{section}.\arabic{theorem}}
\renewcommand{\thelemma}{S.\arabic{section}.\arabic{lemma}}
\renewcommand{\theremark}{S.\arabic{section}.\arabic{remark}}
\renewcommand{\theproposition}{S.\arabic{section}.\arabic{proposition}}
\renewcommand{\thedefinition}{S.\arabic{section}.\arabic{definition}}

This document supplements the main manuscript, `` A Continuous-Time Actor-Critic Framework for Jump-Diffusion Controls with Normalizing Flows'' by collecting theoretical statements, technical derivations and additional details that support the analysis and numerical results in the paper. 

\section{Proof of Theoretical Results}

\subsection{Derivation of the $q$-Function \texorpdfstring{\myeqref{eq:q_function}}{(3.8)}}
\label{appen:derivation_q}
We derive \myeqref{eq:Qderive}, which leads to the definition of the little $q$-function \myeqref{eq:q_function},  via a short-time expansion of $Q_{\delta_t}$ defined in \myeqref{eq:Q_delta}.
Fix $\delta_t>0, t\ge 0$, $\bm x\in\mathbb R^d$, and $\bm u\in\mathcal A$, and recall $(\bm X_s^{\bm u})_{s \geq t}$ from Section~\myref{sec:q_func}.
By the tower property and \(
e^{-\beta(s-t)} = e^{-\beta\delta_t}\,e^{-\beta(s-(t+\delta_t))}\), $Q_{\delta_t}$ can be rewritten as
\begin{equation}
    \begin{aligned}
        Q_{\delta_t}(t,\bm x, \bm u;\pi) &=
\mathbb E\Big[
\int_0^{\delta_t}
  e^{-\beta s}
  f\bigl(t+s,\bm X_{t+s}^{\bm u},\bm u\bigr)\,\mathrm{d}s \\
& +
e^{-\beta\delta_t}\,
J\bigl(t+\delta_t,\bm X_{t+\delta_t}^{\bm u};\pi\bigr)
\;\big|\; \bm X_t^{\bm u}=\bm x
\Big]=: I_1 + I_2 .
    \end{aligned}
\end{equation}

Since $f$ is continuous at $(t, \bm x)$ and $(\bm X_s^{\bm u})_{s \geq t}$ is c\`adl\`ag, by dominated convergence theorem, one has
$\lim_{\delta_t\downarrow 0}\frac{1}{\delta_t}\,
\mathbb E [\int_0^{\delta_t} e^{-\beta s}\,f(t+s,\bm X_{t+s}^{\bm u},\bm u)\,\mathrm ds\ | \bm X_t^{\bm u}=\bm x ]
= f(t,\bm x,\bm u),$
and hence
\begin{equation}\label{eq:I1_first_order}
I_1 =
f(t,\bm x,\bm u)\,\delta_t + o(\delta_t),
\quad (\delta_t\downarrow 0).
\end{equation}
For $I_2$, applying It\^o's formula to  $e^{-\beta(h-t)}J(h,\bm X_h^{\bm u};\pi)$ and integrating over $[t,t+\delta_t)$ gives $e^{-\beta\delta_t}J\bigl(t+\delta_t,\bm X_{t+\delta_t}^{\bm u};\pi\bigr)=
J\bigl(t,\bm X_t^{\bm u};\pi\bigr)+ \int_t^{t+\delta_t} e^{-\beta(h-t)}
\bigl[ \partial_t J\bigl(h,\bm X_{h}^{\bm u};\pi \bigr) + {\mathcal L}^{\bm u} J\bigl(h,\bm X_{h}^{\bm u};\pi \bigr) -\beta J\bigl(h,\bm X_{h}^{\bm u};\pi \bigr) \bigr]\,\mathrm{d}h+
\bigl(M_{t+\delta_t}-M_t\bigr),$
where ${\mathcal L}^{\bm u}$ is the generator associated with $\bm u$, and $(M_h)_{h\ge t}$ is a martingale with $
\mathbb E[M_{t+\delta_t}-M_t | \bm X_t^{\bm u}=\bm x]=0.
$ 
Thus,
\[
\begin{aligned}
I_2
&=
J(t,\bm x;\pi)
+
\mathbb E\Big[
\int_t^{t+\delta_t}
e^{-\beta(h-t)}
\bigl[
\partial_t J\bigl(h,\bm X_{h}^{\bm u};\pi \bigr)
+
\mathcal L^{\bm u} J\bigl(h,\bm X_{h}^{\bm u};\pi \bigr)\\
& \qquad -\beta J\bigl(h,\bm X_{h}^{\bm u};\pi \bigr)
\bigr]\, \mathrm{d}h
\ \big|\ \bm X_t^{\bm u}=\bm x
\Big].
\end{aligned}
\]
By the regularity of $J$ and $\bm X_s^{\bm u}$, as $\delta_t \to 0$ we obtain
\begin{equation}\label{eq:I2_first_order}
I_2
=
J(t,\bm x;\pi)
+
\bigl[
\partial_t J(t,\bm x;\pi)
+\mathcal L^{\bm u} J(t,\bm x;\pi)
-\beta J(t,\bm x;\pi)
\bigr]\delta_t
+
o(\delta_t).
\end{equation}

Combining \eqref{eq:I1_first_order} and \eqref{eq:I2_first_order}, we conclude that
\[
\begin{aligned}
Q_{\delta_t}(t,\bm x, \bm u;\pi)
&=
J(t,\bm x;\pi)
+
\bigl[
\partial_t J(t,\bm x;\pi)
+
f(t,\bm x,\bm u) \\
& \quad +
\mathcal L^{\bm u} J(t,\bm x;\pi)
-
\beta J(t,\bm x;\pi)
\bigr]\delta_t
+
o(\delta_t).
\end{aligned}
\]
Recall the Hamiltonian
$\mathscr H(t,\bm x, \bm u;\pi):= f(t,\bm x,\bm u)+\mathcal L^{\bm u} J(t,\bm x;\pi)$ and the definition
\[
q(t,\bm x, \bm u;\pi)
:=
\lim_{\delta_t\downarrow 0}
\frac{Q_{\delta_t}(t,\bm x, \bm u;\pi)-J(t,\bm x;\pi)}{\delta_t},
\]
we achieve the desired expression \myeqref{eq:q_function}.

\subsection{Proof of Theorem~\texorpdfstring{\myref{thm:PG}}{3.4}}
\label{appen:proof_of_thm2}
\label{appen:proof_lemma2}

We first prove the performance-difference identity \myeqref{eq:perf_diff}, following the idea of \cite[Theorem~2]{zhao2023policy} and adapting it to the present time-inhomogeneous jump-diffusion setting. We then apply this identity to a perturbed policy family and differentiate at the reference parameter. This yields the policy-gradient formula and completes the proof of Theorem~\myref{thm:PG}.

To proceed, we first state the following identity, which follows \cite[Lemma~9]{zhao2023policy}.

\begin{lemma}
\label{lemma:diff_two_policies}
Let $\pi$ and $\hat{\pi}$ be two stochastic policies, $J(\cdot,\cdot;\pi)$ be the value function under $\pi$. Let $\mathcal{L}^\pi$ denote the generator under policy $\pi$.
Then, for all $(t,\bm x)$,
\begin{equation}
\begin{aligned}
\label{eq:gen-identity-two-policies}
&\tilde{f}(t, \bm x, \hat{\pi}) - \tilde{f}(t, \bm x, \pi)
+ \bigl( \mathcal{L}^{\hat{\pi}} - \mathcal{L}^\pi \bigr) J(t, \bm x;\pi)\\
&\quad = \int_{\mathcal{A}}
\big( q(t, \bm x, \bm u; \pi) - \gamma \log\hat\pi(\bm u\mid t,\bm x) \big)\,
\hat{\pi}(\bm u \mid t, \bm x)\mathrm{d} \bm u.
\end{aligned}
\end{equation}
\end{lemma}

\begin{proof}[Proof of performance-difference \myeqref{eq:perf_diff}]
Recall $\mu^{\hat\pi,t,\bm x}$ for the discounted occupation measure induced by~$\hat{\pi}$ starting from $(t,\bm x)$.
Apply Lemma~\myref{lemma:gen_property_shift-final} to  $J(s, \bm y;\pi)$ gives
\begin{equation}
\label{eq:dynkin-Jpi-hatpi}
\int_{[t,\infty)\times \mathbb{R}^d}
\big( -\partial_s J(\cdot,\cdot;\pi) - \mathcal L^{\hat{\pi}} J(\cdot,\cdot;\pi) + \beta J(\cdot,\cdot;\pi) \big)(s, \bm y)
\, \mu^{\hat{\pi},t,\bm x}(\mathrm{d} s,\mathrm{d} \bm y)
= J(t,\bm x;\pi).
\end{equation}
On the other hand, since $J(\cdot,\cdot;\pi)$ is the value function under $\pi$, it satisfies \myeqref{eq:policy-eval-PIDE}.
Integrating it against the measure $\mu^{\hat\pi,t,\bm x}$ gives
\begin{equation}
\label{eq:eval-integrated-hatpi}
0 = \int_{[t,\infty)\times \mathbb{R}^d}
\big( \partial_s J(\cdot,\cdot;\pi) + \mathcal L^\pi J(\cdot,\cdot;\pi)
      + \tilde f(\cdot,\cdot;\pi) - \beta J(\cdot,\cdot;\pi) \big)(s, \bm y)\,
\mu^{\hat{\pi},t,\bm x}(\mathrm{d} s,\mathrm{d} \bm y).
\end{equation}
Adding \eqref{eq:eval-integrated-hatpi} to \eqref{eq:dynkin-Jpi-hatpi} produces
\begin{equation}
\label{eq:Jpi-rep-under-hatpi}
J(t,\bm x;\pi)
= \int_{[t,\infty)\times \mathbb{R}^d}
\big[
   \tilde f(s, \bm y;\pi)
   + \big( \mathcal L^\pi - \mathcal L^{\hat{\pi}} \big) J(s, \bm y;\pi)
\big]\,
\mu^{\hat{\pi},t,\bm x}(\mathrm{d} s,\mathrm{d} \bm y).
\end{equation}
By the definition of $J(t,\bm x;\hat\pi)$ and Lemma~\myref{lem:occupation-identity-inhom-shift}, we also have $J(t,\bm x;\hat\pi)
= \int_{[t,\infty)\times\mathbb R^d}\\
\tilde f(s, \bm y;\hat\pi)\,
\mu^{\hat\pi,t,\bm x}(\mathrm ds,\mathrm d\bm y).$
Subtracting it from \eqref{eq:Jpi-rep-under-hatpi}, 
\begin{equation}
\resizebox{\linewidth}{!}{$
    \begin{aligned}
     J(t,\bm x;\hat\pi) - J(t,\bm x;\pi)
= \int_{[t,\infty)\times\mathbb R^d}
\big(
\tilde f(s, \bm y;\hat\pi) - \tilde f(s, \bm y;\pi)
+ \big(\mathcal L^{\hat\pi}-\mathcal L^\pi\big)J(s, \bm y;\pi)
\big)\,
\mu^{\hat\pi,t,\bm x}(\mathrm ds,\mathrm d\bm y).   
    \end{aligned}
    $}
\end{equation}
By Lemma~\ref{lemma:diff_two_policies},
\[
\resizebox{\linewidth}{!}{$
\begin{aligned}
J(t,\bm x;\hat\pi) - J(t,\bm x;\pi)
&=
\int_{[t,\infty)\times\mathbb R^d}
\int_{\mathcal A}
\big( q(s, \bm y, \bm u;\pi) - \gamma \log\hat\pi(\bm u\mid s,\bm y) \big)\,
\hat\pi(\bm u\mid s,\bm y)\, \mathrm d\bm u
\mu^{\hat\pi,t,\bm x}(\mathrm ds,\mathrm d\bm y).
\end{aligned}
$}
\]
This proves the performance-difference formula.

We now prove the policy-gradient formula. Fix $(t,\bm x)\in[0,\infty)\times\mathbb R^d$ and a reference parameter $\theta_0$.
It suffices to prove for any $h$
\begin{align}\label{eq:pg_integral_version_main_thm2}
&\frac{\mathrm d}{\mathrm d\varepsilon}J(t,\bm x;\pi_{\theta_0+\varepsilon h})\big|_{\varepsilon=0}
=
\int_{[t,\infty)\times\mathbb R^d}\!
\int_{\mathcal A}
\big\langle h,\nabla_\theta \log\pi_\theta(\bm u\mid s,\bm y)\big|_{\theta=\theta_0}\big\rangle\,
\\
& \qquad
\big(q(s,\bm y,\bm u;\pi_{\theta_0})-\gamma\log\pi_{\theta_0}(\bm u\mid s,\bm y)\big) \nonumber \pi_{\theta_0}(\bm u\mid s,\bm y)\,
\mathrm d\bm u\mu^{\theta_0,t,\bm x}(\mathrm ds,\mathrm d\bm y)\,,
\end{align}
then normalizing $\beta\mu^{\theta_0,t,\bm x}$ yields the expectation form in the theorem.
Apply the performance-difference formula with baseline $\pi_{\theta_0}$ and perturbed policy
$\pi_{\theta_0+\varepsilon h}$ yields
\begin{equation}
\label{eq:pdiff_eps}
J(t,\bm x;\pi_{\theta_0+\varepsilon h})-J(t,\bm x;\pi_{\theta_0})
=
\int_{[t,\infty)\times\mathbb R^d}
\upsilon(\varepsilon; s,\bm y)\,
\mu^{\theta_0+\varepsilon h,t,\bm x}(\mathrm ds,\mathrm d\bm y),
\end{equation}
where $\upsilon(\varepsilon; s,\bm y)$ is defined as,
\begin{equation}
\label{eq:def_upsilon_eps}
\upsilon(\varepsilon; s,\bm y)
:=
\int_{\mathcal A}
\pi_{\theta_0+\varepsilon h}(\bm u\mid s,\bm y)
\Big(q(s,\bm y,\bm u;\pi_{\theta_0})-\gamma\log\pi_{\theta_0+\varepsilon h}(\bm u\mid s,\bm y)\Big)\,
\mathrm d\bm u.
\end{equation}
Note that $\upsilon(0;s,\bm y)=0$ for all $(s,\bm y)$ due to the definition \myeqref{eq:q_function} of $q$ and the PDE \myeqref{eq:policy-eval-PIDE} satisfied by $J(\cdot,\cdot;\pi_{\theta_0})$.
Consequently, add and subtract $\mu^{\theta_0,t,\bm x}$ from the right hand side of  \eqref{eq:pdiff_eps} gives:
\begin{align}
 & J(t,\bm x;\pi_{\theta_0+\varepsilon h})-J(t,\bm x;\pi_{\theta_0}) =   \int_{[t,\infty)\times\mathbb R^d}
\big(\upsilon(\varepsilon; s,\bm y)-\upsilon(0; s,\bm y)\big)\,
\mu^{\theta_0+\varepsilon h,t,\bm x}(\mathrm ds,\mathrm d\bm y)\nonumber \\
&=\int_{[t,\infty)\times\mathbb R^d}
\big(\upsilon(\varepsilon; s,\bm y)-\upsilon(0; s,\bm y)\big)\,
\mu^{\theta_0,t,\bm x}(\mathrm ds,\mathrm d\bm y)\nonumber\\
& \quad + \int_{[t,\infty)\times\mathbb R^d}
\big(\upsilon(\varepsilon; s,\bm y)-\upsilon(0; s,\bm y)\big)\,
\big(\mu^{\theta_0+\varepsilon h,t,\bm x}(\mathrm ds,\mathrm d\bm y)-\mu^{\theta_0,t,\bm x}(\mathrm ds,\mathrm d\bm y)\big).\label{eq:pdiff_eps2}
\end{align}

For a direction $h$, write $
\mu_\varepsilon := \mu^{\theta_0+\varepsilon  h,t,\bm x}$, $\mu_0 := \mu^{\theta_0,t,\bm x}$,
and define
\[
g_{\varepsilon,h}(s,\bm y,\bm u) := \pi_{\theta_0+\varepsilon h}(\bm u\mid s,\bm y) ( q(s,\bm y,\bm u;\pi_{\theta_0}) -\gamma\log \pi_{\theta_0+\varepsilon h}(\bm u\mid s,\bm y)).
\]
We impose the following local domination and continuity conditions. For every $h$, 
there exist $\varepsilon_0>0$ and a
nonnegative measurable function
$G_h(s,\bm y,\bm u)$ such that
\[
\Big| \frac{\mathrm d}{\mathrm d\varepsilon} g_{\varepsilon, h}(s,\bm y,\bm u) \Big| \leq G_h(s,\bm y,\bm u), \qquad \forall |\varepsilon|\leq\varepsilon_0   \text{ and a.e. } (s,\bm y,\bm u).
\]
Define $ w_h(s,\bm y) := \int_{\mathcal A} G_h(s,\bm y,\bm u)\,\mathrm d\bm u$, and assume that $
\int_{[t,\infty)\times\mathbb R^d} w_h(s,\bm y)\, \mu_0(\mathrm ds,\mathrm d\bm y) <\infty$. Finally, assume that
\[
\|\mu_\varepsilon-\mu_0\|_{w_h} \to 0 \qquad\text{as }\varepsilon \to 0,
\]
where, for a finite signed measure $\eta$, $ \|\eta\|_{w_h} := \int_{[t,\infty)\times\mathbb R^d} w_h(s,\bm y)\, |\eta|(\mathrm ds,\mathrm d\bm y)$. Under the preceding assumptions, differentiation under the integral
in~\eqref{eq:def_upsilon_eps} is justified:
\begin{align}
\frac{\mathrm d}{\mathrm d\varepsilon}\upsilon(\varepsilon;s,\bm y)\big|_{\varepsilon=0} =&
\int_{\mathcal A}
\big\langle h,\nabla_\theta \pi_\theta(\bm u \mid s, \bm y)\big|_{\theta=\theta_0}\big\rangle
\big(q(s, y,\bm u; \pi_{\theta_0})-\gamma\log\pi_{\theta_0}(\bm u \mid s, \bm y)\big)\,\mathrm d\bm u \nonumber\\
& -\gamma\int_{\mathcal A}\pi_{\theta_0}(\bm u \mid s, \bm y )\big\langle h,\nabla_\theta \log\pi_\theta(\bm u \mid s, \bm y)\big|_{\theta=\theta_0}\big\rangle\,\mathrm d\bm u .
\label{eq:upsilon_deriv_intermediate}
\end{align}
Using $\nabla_\theta \pi_\theta=\pi_\theta\nabla_\theta\log\pi_\theta$ and the normalization
$\int_{\mathcal A}\pi_\theta(\bm u)\,\mathrm d\bm u=1$ (hence $\int_{\mathcal A}\nabla_\theta\pi_\theta(\bm u)\,\\ \mathrm d\bm u=0$), the second term vanishes and \eqref{eq:upsilon_deriv_intermediate} can be written as
\begin{equation}
\begin{aligned}
    &\frac{\mathrm d}{\mathrm d\varepsilon}\upsilon(\varepsilon;s,\bm y)\big|_{\varepsilon=0} =
\int_{\mathcal A}
\big\langle h,\nabla_\theta \log\pi_\theta(\bm u \mid s, \bm y)\big|_{\theta=\theta_0}\big\rangle\\
&\quad \big(q(s, y,\bm u; \pi_{\theta_0})-\gamma\log\pi_{\theta_0}(\bm u \mid s, \bm y)\big)\pi_{\theta_0}(\bm u \mid s, \bm y)\,\mathrm d\bm u.
\label{eq:upsilon_deriv_final}
\end{aligned}
\end{equation}

Now, define $ D_\varepsilon(s,\bm y) := \frac{ \upsilon(\varepsilon;s,\bm y) -\upsilon(0;s,\bm y)} {\varepsilon}$. By the mean-value theorem and the domination assumption, one has $|D_\varepsilon(s,\bm y)| \leq w_{\bm h}(s,\bm y)$, $0<|\varepsilon|\leq\varepsilon_0$,
and $D_\varepsilon(s,\bm y) \to \frac{\mathrm d}{\mathrm d\varepsilon}\upsilon(\varepsilon;s,\bm y)\big|_{\varepsilon=0}$
pointwise. Then, by dominated convergence,
\[
\lim_{\varepsilon\to0} \int_{[t,\infty)\times\mathbb R^d} D_\varepsilon(s,\bm y)\, \mu_0(\mathrm ds,\mathrm d\bm y) = \int_{[t,\infty)\times\mathbb R^d} \frac{\mathrm d}{\mathrm d\varepsilon}\upsilon(\varepsilon;s,\bm y)\big|_{\varepsilon=0}\, \mu_0(\mathrm ds,\mathrm d\bm y).
\]
Moreover,
\begin{align*}
\Big|
\int_{[t,\infty)\times\mathbb R^d}
D_\varepsilon(s,\bm y)\,
(\mu_\varepsilon-\mu_0)(\mathrm ds,\mathrm d\bm y)
\Big|
&\leq
\int_{[t,\infty)\times\mathbb R^d}
w_{ h}(s,\bm y)\,
|\mu_\varepsilon-\mu_0|(\mathrm ds,\mathrm d\bm y)\\
&=
\|\mu_\varepsilon-\mu_0\|_{w_{ h}}
\to 0.
\end{align*}

Dividing~(S.1.12) by $\varepsilon$ and letting $\varepsilon \to 0$ therefore gives \eqref{eq:pg_integral_version_main_thm2}.
Since this holds for every direction $h$, it gives the
G\^ateaux derivative of $J(t,\bm x;\pi_\theta)$ at $\theta_0$.
Using that $\beta\mu^{\theta_0,t,\bm x}$ is a probability measure
then yields the policy-gradient representation~(3.12).

\end{proof}

\subsection{Derivation of Lemma~\texorpdfstring{\myref{lem:gae-approximation}}{3.5}}
\label{appen:derivation_gae}

\begin{proof}
Fix $\delta_t>0, t\ge 0$, $\bm x\in\mathbb R^d$, and $\bm u\in\mathcal A$, and recall $(\bm X_s^{\bm u})_{s \geq t}$ from Section~\myref{sec:q_func}. Apply the It\^o-L\'evy formula to
$e^{-\beta(s-t)}J(s,\bm X_s^{\bm u};\pi)$ on $[t,t+\delta_t)$, we obtain
\begin{equation}
    \begin{aligned}
        &e^{-\beta\delta_t}J\bigl(t+\delta_t, \bm X_{t+\delta_t}^{\bm u};\pi\bigr) - J\bigl(t, \bm X_t^{\bm u};\pi\bigr)\\
        &\quad =
\int_t^{t+\delta_t} e^{-\beta(s-t)}
\bigl( \partial_s J + \mathcal L^{\bm u} J - \beta J \bigr)(s,\bm X_s^{\bm u};\pi)\,\mathrm{d}s
\;+\; \tilde M_{t+\delta_t}- \tilde M_t,
    \end{aligned}
\end{equation}
where $\tilde M$ collects the stochastic integrals with respect to the Brownian motion and the compensated Poisson random measure.
Taking conditional expectation with respect to $\mathcal F_t$ and using the integrability assumption in Lemma~\myref{lem:gae-approximation} (which ensures the local martingale term is a true martingale on $[t,t+\delta_t)$), we get
\begin{equation}
    \begin{aligned}
       & \mathbb E\big[
e^{-\beta\delta_t}J\bigl(t+\delta_t, \bm X_{t+\delta_t}^{\bm u};\pi\bigr) - J\bigl(t, \bm X_t^{\bm u};\pi\bigr)
\,|\,\mathcal F_t
\big]\\
&\quad =
\mathbb E\Big[
\int_t^{t+\delta_t} e^{-\beta(s-t)}
\bigl( \partial_s J + \mathcal L^{\bm u} J - \beta J \bigr)(s,\bm X_s^{\bm u};\pi)\,\mathrm{d}s
\,\big|\,\mathcal F_t
\Big].
    \end{aligned}
\end{equation}
By continuity of $J$ and the coefficients, together with dominated convergence (using the integrability assumption in Lemma~\myref{lem:gae-approximation}), as $\delta_t\to 0$,
\begin{align*}
\resizebox{\linewidth}{!}{$
\mathbb E\Big[
\int_t^{t+\delta_t} e^{-\beta(s-t)}
\bigl( \partial_s J + \mathcal L^{\bm u} J - \beta J \bigr)(s,\bm X_s^{\bm u};\pi)\,\mathrm{d}s
\,\big|\,\mathcal F_t
\Big]
=
\bigl(\partial_t J + \mathcal L^{\bm u} J - \beta J\bigr)(t,\bm X_t^{\bm u};\pi)\,\delta_t
+ o(\delta_t).
$}
\end{align*}

Since $f(t,\bm X_t^{\bm u},\bm u)$ is $\mathcal F_t$-measurable, one has
\begin{equation}
    \begin{aligned}
      &  \mathbb E \big[
f(t,\bm X_t^{\bm u},\bm u)\,\delta_t
+ e^{-\beta\delta_t}J\bigl(t+\delta_t, \bm X_{t+\delta_t}^{\bm u};\pi\bigr)
- J\bigl(t,\bm X_t^{\bm u};\pi\bigr)
|\mathcal F_t
\big] \\
&\quad =
\bigl(\partial_t J + \mathcal L^{\bm u} J - \beta J + f\bigr)(t,\bm X_t^{\bm u}, \bm u;\pi)\,\delta_t
+ o(\delta_t).
    \end{aligned}
\end{equation}
Dividing it by $\delta_t$ and using \myeqref{eq:gae_def_J} and the definition \myeqref{eq:q_function} of $q$-function gives \myeqref{eq:gae-first-order_J}.
\end{proof}

\subsection{Proof for \texorpdfstring{\myeqref{eq:V_pinn}}{(4.10)}}
\begin{proof}
Let $\mathscr H(u):=\mathscr H(x,u,\nabla_xV,\nabla_x^2V),\, Z(x):=\int_{\mathcal A}\exp\!\left(\frac{1}{\gamma}\mathscr H(\tilde u)\right)\,d\tilde u.$
By the Gibbs representation of the optimal policy, $\pi^*(u\mid x)=\frac{\exp\!\left(\frac{1}{\gamma}\mathscr H(u)\right)}{Z(x)}.$
Hence $\log \pi^*(u\mid x)=\frac{1}{\gamma}\mathscr H(u)-\log Z(x),$
which implies $\mathscr H(u)-\gamma\log\pi^*(u\mid x)=\gamma\log Z(x).$
For the entropy-regularized HJB equation, 
\[
\beta V(x)
=
\int_{\mathcal A}
\bigl(\mathscr H(u)-\gamma\log\pi^*(u\mid x)\bigr)\pi^*(u\mid x)\,du\,,
\]
substituting $\pi^*$ into HJB and using the normalization
$\int_{\mathcal A}\pi^*(u\mid x)\,du=1$ yields
\[
\beta V(x)
=
\gamma\log Z(x)
=
\gamma\log\int_{\mathcal A}
\exp\!\left(\frac{1}{\gamma}\mathscr H(x,u,\nabla_xV,\nabla_x^2V)\right)\,du.
\]
which is exactly \myeqref{eq:V_pinn}.
\end{proof}

\section{Proof For Benchmarks}\label{sec:benchmark}

\subsection{Closed-form Solution in the LQ Case}
\label{sec:lq_closed_form}

For the linear-quadratic problem, we adopt the quadratic ansatz
\begin{equation}
\label{eqn:lq_ansatz_unified}
V(t,\bm x)=\bm x^\top \bm H(t)\bm x+g_\gamma(t),
\qquad
\bm H(t)\in\mathbb S^{d},\quad g_\gamma(t)\in\mathbb R.
\end{equation}
We show that, under this ansatz, the entropy-regularized HJB admits an explicit Gaussian optimizer and reduces to a Riccati--scalar ODE system.

\begin{proof}
Recall that $f(t,\bm x,\bm u)=-(\bm u^\top \bm R(t)\bm u+\bm x^\top \bm Q(t)\bm x),
\, \bm R(t)\succ \bm 0,\ \bm Q(t)\succeq \bm 0.$
Substituting \eqref{eqn:lq_ansatz_unified} into the entropy-regularized HJB, the policy-dependent part becomes
\[
\sup_{\pi(\cdot\mid t,\bm x)}
\left\{
\mathbb E_\pi[-\bm u^\top \bm R(t)\bm u
+2\,\bm u^\top \bm B(t)^\top \bm H(t)\bm x]
+\gamma\,\mathbb E_\pi[-\log\pi(\bm u\mid t,\bm x)]
\right\},
\]
By the optimizer is of Gibbs form, $\pi^*(\bm u\mid t,\bm x)\propto
\exp\!\left(\frac{\mathscr H(t,\bm x,\bm u)}{\gamma}\right) =\\
\mathcal N\!\left(
\bm R(t)^{-1}\bm B(t)^\top \bm H(t)\bm x,\,
\frac{\gamma}{2}\bm R(t)^{-1}
\right)$, and $\mathbb E_{\pi^*}[\bm u\mid t,\bm x] =\bm R(t)^{-1}\bm B(t)^\top \bm H(t)\bm x,
\,\,\\
\operatorname{Var}_{\pi^*}(\bm u\mid t,\bm x) = \frac{\gamma}{2}\bm R(t)^{-1}$, matching the coefficients in HJB, we obtain
\begin{equation}
\label{eq:lq_riccati_scalar_inhomo_H}
\bm H'(t)
=
\beta \bm H(t)+\bm Q(t)-\bm H(t)\bm B(t)\bm R(t)^{-1}\bm B(t)^\top \bm H(t),
\end{equation}
and
\begin{equation}
\label{eq:lq_riccati_scalar_inhomo_g}
g_\gamma'(t) =
\beta g_\gamma(t)
-\operatorname{Tr}\!\big(\bm\Sigma(t)\bm\Sigma(t)^\top \bm H(t)\big)
-\operatorname{Tr}\!\big(\bm\Lambda(t)\,\operatorname{diag}(\bm\alpha(t))\,\bm H(t)\,\operatorname{diag}(\bm\alpha(t))\big)
-c_\gamma(t),
\end{equation}
where
\[
c_\gamma(t)
=
\frac{\gamma}{2}\bigl(m\log(\pi\gamma)-\log\det\bm R(t)\bigr),
\qquad
\bm\Lambda(t):=\operatorname{diag}(\lambda_i(t)).
\]
Hence the value function is obtained.

In the standard case $\gamma=0$, the entropy term disappears and the optimal stochastic policy collapses to the deterministic feedback control
\begin{equation}
\label{eqn:lq_u_star_convergent_merged}
\bm u^*(t,\bm x)=\bm R(t)^{-1}\bm B(t)^\top \bm H(t)\bm x,
\end{equation}
which is exactly the mean of optimal policy when $\gamma>0$, and it can be obtained directly from the first-order optimality condition. Accordingly, \eqref{eq:lq_riccati_scalar_inhomo_H} remains unchanged, while \eqref{eq:lq_riccati_scalar_inhomo_g} reduces to the same scalar equation without the entropy correction term \(c_\gamma(t)\).

Finally, the ODE system above determines the candidate solution, and the admissible branch is selected by the coefficient class. If the coefficients converge as \(t\to\infty\), we impose
\[
\lim_{t\to\infty}\bm H(t)=\bm H_\infty,
\qquad
\lim_{t\to\infty}g_\gamma(t)=g_\infty^\gamma,
\]
where \(\bm H_\infty\) solves the limiting algebraic Riccati equation and \(g_\infty^\gamma\) is determined by the corresponding stationary scalar balance. To implement this asymptotic condition in the one-dimensional setting, we set $\bm H(T_\infty)=\bm H_\infty<0$, choosing the stabilizing solution of the limiting algebraic Riccati equation, and then integrate the Riccati equation backward. For the parameter values used in our experiments, the computed solution satisfies $\bm H(t) \prec 0$ at every time-grid point. If the coefficients are \(P\)-periodic, then we impose the periodic boundary condition
\[
\bm H(t+P)=\bm H(t),\qquad g_\gamma(t+P)=g_\gamma(t).
\]
Numerically, the value of $\bm H(P)=\bm H(0)$ is determined by a shooting method. For the parameters used, the
resulting periodic solution satisfies $\bm H(t) \prec 0$ throughout $[0,P]$.

In the time-homogeneous case, \(\bm H'(t)=0\) and \(g_\gamma'(t)=0\), so the system reduces to the associated algebraic Riccati equation together with the stationary scalar equation.
\end{proof}

\subsection{Closed-form Solution of  Multi-agent Game}
\label{appen:proof_multi-agent_game}

\begin{proof}
Recall that $Y_t^i := \frac1n\sum_{j\neq i} X_t^j$ represents the average wealth of other agents, and that $V^i(t, x, y)$ is defined in (4.14).
Then $Y_t^i$ satisfies 
\begin{equation}
\begin{aligned}
\mathrm{d}Y_t^i
&= \widehat{ub}\,\mathrm{d}t
  + \widehat{u\sigma}\,\mathrm{d}W_t^0
  + \frac1n \sum_{j\neq i} u_j\eta_j\,\mathrm{d} W_t^j
  + \frac1n \sum_{j\neq i} u_j\alpha_j\,\mathrm{d}M_t ^j
  + \widehat{u\xi}\,\mathrm{d}M_t ^0,
\end{aligned}
\label{eq:SDE-Y-i}
\end{equation}
where $\widehat{ub}
\!:= \frac1n\sum_{j\neq i} u_j b_j$,
$\widehat{u\sigma}
\!:= \frac1n\sum_{j\neq i} u_j\sigma_j$,
$\widehat{u^2\eta^2}\!:= \frac1n\sum_{j\neq i} u_j^2\eta_j^2$, \\
$
\widehat{u\xi}
\!:= \frac1n\sum_{j\neq i} u_j\xi_j$.
By the dynamic programming principle, $V^i$ satisfies the HJB equation
\begin{equation}
\partial_t V^i
+ \sup_{u_i\in\mathbb R}
\bigl\{
\mathcal L^{i,u_i} V^i + f_i
\bigr\}
- \beta V^i
= 0,
\label{eq:HJB-general-CARA}
\end{equation}
where $\mathcal L^{i,u_i}$ is the generator of the pair
$(X_t^i,Y_t^i)$ under the control $u_i$ for agent $i$, with the controls 
$u_j$, $j\neq i$ held fixed. A direct computation yields
\begin{align}
&\mathcal L^{i,u_i} V^i(t,x,y)\\
&=
\bigl[
u_i b_i\,V_x^i
+ \widehat{ub}\,V_y^i
+ \tfrac12(\eta_i^2+\sigma_i^2)u_i^2\,V_{xx}^i
+ \tfrac12\bigl((\widehat{u\sigma})^2 + \tfrac1n\widehat{u^2\eta^2}\bigr)V_{yy}^i
+ u_i\sigma_i\widehat{u\sigma}\,V_{xy}^i
\bigr]_{(t,x,y)}
\notag\\
&\quad
+ \lambda_i\bigl(
V^i(t,x+u_i\alpha_i,y) - V^i(t,x,y)
- u_i\alpha_i\,V_x^i(t,x,y)
\bigr)
\notag\\
&\quad
+ \sum_{j\neq i}\lambda_j\bigl(
V^i\bigl(t,x,y+\tfrac{u_j\alpha_j}{n}\bigr)
- V^i(t,x,y)
- \tfrac{u_j\alpha_j}{n}\,V_y^i(t,x,y)
\bigr)
\notag\\
&\quad
+ \lambda_0\bigl(
V^i(t,x+u_i\xi_i,y+\widehat{u\xi})
- V^i(t,x,y)
- u_i\xi_i\,V_x^i(t,x,y)
- \widehat{u\xi}\,V_y^i(t,x,y)
\bigr).
\label{eq:generator-i-u}
\end{align}

We seek solutions of the form
\begin{equation}
V^i(t,x,y)
= -\frac{1}{K_i(t)}\exp\big(
-\frac{1}{\varrho_i}
\bigl((1 - \frac{\varpi_i}{n})x -\varpi_i y\bigr)
\big).\label{eq:ansatz-CARA}
\end{equation}
Define $\chi_i 
 = \frac{1-\varpi_i/n}{\varrho_i}$, and 
$ \rho_i = \frac{\varpi_i}{\varrho_i}$. Substituting the ansatz into \eqref{eq:generator-i-u} and dividing by $V^i<0$,
the $u_i$-dependent terms from the drift and diffusion are
$
\frac{1}{V^i}\Bigl(
u_i b_i V_x^i
+ \tfrac12(\eta_i^2+\sigma_i^2)u_i^2 V_{xx}^i
+ u_i\sigma_i\widehat{u\sigma}\,V_{xy}^i
\Bigr)
=
-\chi_i b_i u_i
+ \tfrac12\chi_i^2(\eta_i^2+\sigma_i^2)u_i^2
- \chi_i\rho_i\sigma_i\widehat{u\sigma}\,u_i.
$
The jump terms contribute 
$
\frac{\lambda_i}{V^i}
\big(V^i(t,x+u_i\alpha_i,y)-V^i(t,x,y)-u_i\alpha_i V_x^i\big)
=
\lambda_i\bigl(e^{-\chi_i\alpha_i u_i}-1+\chi_i\alpha_i u_i\bigr)$,
and $\frac{\lambda_0}{V^i}\big(
V^i(t,x+u_i\xi_i,y+\widehat{u\xi}) - V^i(t,x,y) - u_i\xi_i V_x^i - \widehat{u\xi} V_y^i \big)
=
\lambda_0\big(
e^{-\chi_i\xi_i u_i + \rho_i\widehat{u\xi}}
-1 + \chi_i\xi_i u_i - \rho_i\widehat{u\xi}
\big).$
Collecting all $u_i$-dependent terms defines the function  $\Psi_i$:
\begin{equation}
\begin{aligned}
    \Psi_i(u)
&=
-\chi_i b_i u
+ \frac12\chi_i^2(\eta_i^2+\sigma_i^2)u^2
- \chi_i\rho_i\sigma_i\widehat{u\sigma}\,u
+ \lambda_i\bigl(e^{-\chi_i\alpha_i u}-1+\chi_i\alpha_i u\bigr)\\
&\quad + \lambda_0\Bigl(
e^{-\chi_i\xi_i u + \rho_i\widehat{u\xi}} + \chi_i\xi_i u
\Bigr).
\label{eqn:Psi-def-CARA-correct-u}
\end{aligned}
\end{equation}
The first order condition $\Psi_i'(u) = 0$ gives a candidate optimal control for agent $i$. Since $\Psi_i$ is strictly convex in $u$, its unique minimizer $u_i^*$ is characterized by $\Psi_i'(u_i^*)=0$. Because the $u$-dependent term in agent~$i$'s HJB equation is $V_i\Psi_i(u)$ and $V_i<0$, minimizing $\Psi_i$ is equivalent to maximizing the HJB expression over $u$. Consequently, a collection of feedback controls $\bm u^*=(u_1^*,\dots,u_n^*)$ forms a Markovian Nash equilibrium if and only if it solves the coupled system $\Psi_i'(u_i^\ast) = 0$, $i = 1, 2, \ldots, n$. This proves the equilibrium characterization.

Matching the coefficients and using time homogeneity gives
$ \Psi_i(u_i^*) + K_i(t) + C_i - \beta = 0$, where $C_i$ is independent of $u_i$ and given by
\begin{equation}
\begin{aligned}
    C_i&=
\rho_i \,\widehat{ub}
+ \frac{1}{2}\,\rho_i^2\!
  \big(
    \frac{1}{n^2}\sum_{j\neq i}(u_j \eta_j)^2
    + \big(\frac{1}{n}\sum_{j\neq i} u_j \sigma_j\big)^2
  \big) \\
&+ \sum_{j\neq i}\lambda_j\!\big(
      \exp\!\big(\frac{\varpi_i}{\varrho_i}\,\frac{u_j \alpha_j}{n}\big) - 1
      - \frac{\varpi_i}{\varrho_i}\,\frac{u_j \alpha_j}{n}
  \big)
- \lambda_0 - \lambda_0\,\rho_i\,\widehat{u\xi}.
\end{aligned}
\end{equation}
Hence $K_i^* = \beta - \Psi_i(u_i^*) - C_i$, and the value function of agent $i$ is 
\begin{equation}
V^i(x,y)
= -\frac{1}{\beta - \Lambda_i^*}
\exp\left(
-\frac{1}{\varrho_i}
\Bigl(\bigl(1-\tfrac{\varpi_i}{n}\bigr)x - \varpi_i y\Bigr)
\right),
\end{equation}
where $\Lambda_i^* := \Psi_i(u_i^*) + C_i,$ and $\beta > \Lambda_i^*$ ensures the concavity.
\end{proof}

\section{More Details For Experiments}

\subsection{Numerical Simulation of State Dynamics}
\label{sec:discretization}

In all experiments we simulate the controlled jump-diffusion dynamics by an Euler--Maruyama scheme on a uniform grid $0=t_0<\cdots<t_K$ with $t_k:=k\delta_t$.  
We do not directly discretize the averaged coefficients of the exploratory SDE~(2.2). Instead, at each grid point $t_k$, we sample actions $\bm u_k^{(\ell)}\sim\pi_\theta(\cdot\mid t_k,\bm X_k^{(\ell)})$, hold this action fixed over $[t_k,t_{k+1})$, and evolve the resulting strict controlled jump-diffusion. Under the assumptions of \cite{bender2026continuous}, this grid-sampled process converges weakly to the exploratory SDE~(2.2) as $\delta_t \downarrow 0$.

At each $t_k$, we sample an action $\bm u_k^{(\ell)}\sim\pi_\theta(\cdot\mid t_k,\bm X_k^{(\ell)})$, a Brownian increment $\Delta\bm W_k^{(\ell)}\sim\mathcal N(\bm 0,\delta_t\bm I_d)$, and, under the Poisson jump specification, jump counts $\Delta N_k^{(i,\ell)}\sim\mathrm{Poisson}(\lambda_i\delta_t)$ independently across $i,k,\ell$, and update
\begin{equation}
\label{eq:Xdiscrete}
\begin{aligned}
\bm X_{k+1}^{(\ell)}
&= \bm X_k^{(\ell)}
+ \bm b(t_k,\bm X_k^{(\ell)},\bm u_k^{(\ell)})\,\delta_t
+ \bm\sigma(t_k,\bm X_k^{(\ell)},\bm u_k^{(\ell)})\,\Delta \bm W_k^{(\ell)}\\
&\quad
+ \sum_{i=1}^d
\bm\alpha(t_k,\bm X_k^{(\ell)},\bm u_k^{(\ell)},\bm e_i)
\,\Delta M_k^{(i,\ell)},
\end{aligned}
\end{equation}
where $\Delta M_k^{(i,\ell)}:=\Delta N_k^{(i,\ell)}-\lambda_i\delta_t$ denotes the compensated jump increments. These realized Brownian increments $\Delta\bm W_k^{(\ell)}$ and jumps are also used in the martingale-corrected TD error in the main text.

\subsection{Discretization of Non-local Term}
\label{sec:mart_disc}

We provide here the discretization of the martingale correction $\mathcal I^\pi_{t,t+\delta_t}$ defined in \myeqref{eq:mart_inc_J_corrected}. This correction is subtracted from the one-step standard TD error to obtain the martingale-corrected critic update. For notational simplicity, we describe the discretization along one trajectory over the interval $[t_k,t_{k+1}]$.

The Brownian component on $[t_k,t_{k+1}]$ is approximated by 
\(\bigl(\nabla_{\bm x}V_\psi(t_k,\bm X_k)^\top\\\bm\sigma(t_k,\bm X_k,\bm u_k)\bigr)\Delta\bm W_k.\) The jump component is

\[
\mathcal I^J_k:=\int_{t_k}^{t_{k+1}}\int_{\R^d}e^{-\beta(s-t_k)}\Delta_{\bm z}V_\psi(s,\bm X_{s-},\bm u_s)\,\widetilde N(\mathrm ds,\mathrm d\bm z),
\]
where
\[
\Delta_{\bm z}V_\psi(t,\bm x,\bm u):=V_\psi\bigl(t,\bm x+\bm\alpha(t,\bm x,\bm u,\bm z)\bigr)-V_\psi(t,\bm x).
\]
For a finite-activity jump measure, using $\widetilde N(\mathrm ds,\mathrm d\bm z)=N(\mathrm ds,\mathrm d\bm z)-\nu(\mathrm d\bm z)\,\mathrm ds$, we obtain
\begin{align}
\mathcal I^J_k=\sum_{\tau_i\in(t_k,t_{k+1}]}e^{-\beta(\tau_i-t_k)}\Delta_{\bm z_i}V_\psi(\tau_i,\bm X_{\tau_i-},\bm u_{\tau_i}) -\int_{t_k}^{t_{k+1}}e^{-\beta(s-t_k)}V_{\mathrm{non}}(s,\bm X_{s-},\bm u_s)\,\mathrm ds,
\label{eq:jump-martingale-decomposition}
\end{align}
where $\tau_i$ and $\bm z_i$ denote the realized jump times and jump marks, respectively, and
\begin{equation}
V_{\mathrm{non}}(t,\bm x,\bm u):=\int_{\R^d}\Delta_{\bm z}V_\psi(t,\bm x,\bm u)\,\nu(\mathrm d\bm z)
\label{eq:Vnon-supp}
\end{equation}
is the instantaneous compensator rate of the jump increments of the critic. Applying the same left-endpoint approximation to \eqref{eq:jump-martingale-decomposition} gives
\begin{equation*}
\mathcal I^J_k\approx\sum_{i=N_{t_k}+1}^{N_{t_{k+1}}}\left[V_\psi\bigl(t_k,\bm X_k+\bm\alpha(t_k,\bm X_k,\bm u_k,\bm z_i)\bigr)-V_\psi(t_k,\bm X_k)\right]-\delta_tV_{\mathrm{non}}(t_k,\bm X_k,\bm u_k).
\end{equation*}

For the discrete L\'evy measure used in our numerical experiments,
\[
\nu(\mathrm d\bm z)=\sum_{i=1}^{d}\lambda_i\delta_{\bm e_i}(\mathrm d\bm z),
\]

the compensator rate in \eqref{eq:Vnon-supp} can be evaluated exactly as
\begin{equation}
V_{\mathrm{non}}(t,\bm x,\bm u)=\sum_{i=1}^{d}\lambda_i\left[V_\psi\bigl(t,\bm x+\bm\alpha(t,\bm x,\bm u,\bm e_i)\bigr)-V_\psi(t,\bm x)\right].
\label{eq:Vnon-discrete}
\end{equation}

More generally, suppose that $\nu$ is known and has finite total mass $\lambda_\nu:=\nu(\R^d)<\infty$, and define $\overline\nu(\mathrm d\bm z):=\nu(\mathrm d\bm z)/\lambda_\nu$. Then

\[
V_{\mathrm{non}}(t,\bm x,\bm u)=\lambda_\nu\mathbb E_{\bm Z\sim\overline\nu}\left[\Delta_{\bm Z}V_\psi(t,\bm x,\bm u)\right].
\]

Given independent samples $\bm Z_1,\ldots,\bm Z_M\sim\overline\nu$, it may be estimated by
\begin{equation*}
\frac{\lambda_\nu}{M}\sum_{j=1}^{M}\left[V_\psi\bigl(t,\bm x+\bm\alpha(t,\bm x,\bm u,\bm Z_j)\bigr)-V_\psi(t,\bm x)\right].
\end{equation*}
Alternatively, when the jump-mark dimension is sufficiently small, the integral in \eqref{eq:Vnon-supp} may be evaluated by numerical quadrature.

Alternatively, one may introduce
an auxiliary neural network $V_{\mathrm{non},\eta}(t,\bm x,\bm u)$ to approximate the compensator rate. Since the corresponding compensated jump integral is a martingale, the network may be trained through the conditional moment condition

\[
\mathbb E\!\Big[
\sum_{\tau_i\in(t_k,t_{k+1}]}
\Delta_{\bm z_i}V_\psi(t_k,\bm X_k,\bm u_k)
-\delta_tV_{\mathrm{non},\eta}(t_k,\bm X_k,\bm u_k)
\,\mid \,\mathcal F_{t_k}
\Big]\approx0,
\]

In practice, this condition can be enforced by minimizing the empirical squared residual over sampled trajectories. The resulting approximation can then be substituted for $V_{\mathrm{non}}$ in the martingale-corrected TD error, thereby avoiding repeated evaluation of the nonlocal integral. A rigorous justification is beyond the scope of this work.

We also emphasize that $V_{\mathrm{non}}$ is needed only for the optional martingale-corrected TD update and is not required for the standard TD critic update.

\subsection{Numerical Comparison of Usual and Martingale-Corrected TD Errors}

To examine the variance-reduction effect of the martingale correction, we compare the usual TD error, implemented with the semi-gradient convention described in the main text, with its martingale-corrected counterpart on the five-dimensional jump-diffusion LQ problem. Since the martingale-corrected TD error requires the realized Brownian and jump increments, the purpose of this experiment is not to provide an information-matched comparison between two RL algorithms, but rather to isolate the effect of subtracting the martingale fluctuations.

We consider two experiments. First, fixing the analytic LQ solution, we evaluate both the usual and martingale-corrected TD residuals on the same states and using the same Brownian and Poisson increments. For each $\delta_t$, we generate $30$ independent minibatches with $L=1{,}000$ trajectories. The ratio $\frac{\operatorname{Var}(\tilde \delta_{\mathrm{TD}})}{\operatorname{Var}(\delta_{\mathrm{TD}})}$ decreases from $5.35\times10^{-4}$ at $\delta_t=0.02$ to $9.34\times10^{-6}$ at $\delta_t=0.002$. We observe a similar reduction in the variability of the corresponding squared one-step losses. This directly confirms the variance-reduction effect of subtracting the martingale increments.

Second, we compare the two critic losses within the complete actor--critic training procedure. All algorithmic settings are kept fixed, and only the critic loss is changed from the usual TD error to the martingale-corrected TD error. We perform the complete training with three paired random seeds, $2025$, $2026$, and $42$. This comparison follows the diffusion-only study of Zhou, Han, and Lu~\cite{zhou2021actor}, which examines the effect of martingale correction on value-function and policy errors. Our jump-diffusion setting additionally includes the compensated Poisson martingale term.
For each trained model, we evaluate the learned value function and policy on $L_{\mathrm{eval}}=500$ common reference trajectories at $K+1=501$ time points. At time $t_k$, the value-function error is
\[
\operatorname{MSE}_V(t_k)
=
\frac{1}{L_{\mathrm{eval}}}
\sum_{\ell=1}^{L_{\mathrm{eval}}}
\left[
\widehat V(t_k,\bm X_k^{(\ell)})
-
V(t_k,\bm X_k^{(\ell)})
\right]^2,
\]
while the policy error is measured by the root-mean-square Gaussian $W_2$ distance
\[
D^{W_2}_\pi(t_k)
=
\left\{
\frac{1}{L_{\mathrm{eval}}}
\sum_{\ell=1}^{L_{\mathrm{eval}}}
\left(
\left\|
\widehat{\bar{\bm\mu}}_\ell(t_k)-\bar{\bm\mu}^{*}_\ell(t_k)
\right\|_2^2
+
\left\|
\widehat{\operatorname{Std}}_\ell(t_k)-\operatorname{Std}^{*}_\ell(t_k)
\right\|_2^2
\right)
\right\}^{1/2}.
\]
Here, $\overline{\operatorname{MSE}}_V$ and $\overline D^{W_2}_\pi$ denote the arithmetic averages of $\operatorname{MSE}_V(t_k)$ and $D^{W_2}_\pi(t_k)$, respectively, over all $501$ time points. 
Following Eq.~(4.3) under $d=m=5$ and $T=1$, we use the same parameter setting in Fig.~1 for the following tests. 

Table~\ref{tab:trajectory-results} reports the time-averaged value-function and policy errors. Each entry in Table~\ref{tab:trajectory-results} is the mean $\pm$ sample standard deviation across the three seed-level results ($n=3$); each seed-level result itself is evaluated using the same $500$ reference trajectories. The time-averaged value-function MSE is reduced by $99.07\%$, while the time-averaged Gaussian policy-distribution distance is reduced by $74.10\%$, which indicates that the martingale-corrected TD update yields substantially smaller errors in both the learned value function and the policy distribution. 

\begin{table}[htbp]
  \centering
  \caption{Time-averaged value-function MSE and Gaussian policy-distribution distance. 
  }
  \label{tab:trajectory-results}
  \setlength{\tabcolsep}{6pt}
  \resizebox{0.80\textwidth}{!}{%
  \begin{tabular}{@{}lcc@{}}
    \toprule
    Critic loss
    & $\overline{\operatorname{MSE}}_V$
    & $\overline D^{W_2}_\pi$ \\
    \midrule
    Usual TD error
    & $(4.2792\pm4.0007)\times10^{-1}$
    & $(3.8409\pm0.1306)\times10^{-2}$ \\
    Martingale-corrected TD error
    & $(3.9824\pm3.2527)\times10^{-3}$
    & $(9.9493\pm2.3502)\times10^{-3}$ \\
    \bottomrule
  \end{tabular}
  }
\end{table}

\begin{figure}[htbp]
  \centering
  \includegraphics[width=0.80\textwidth]{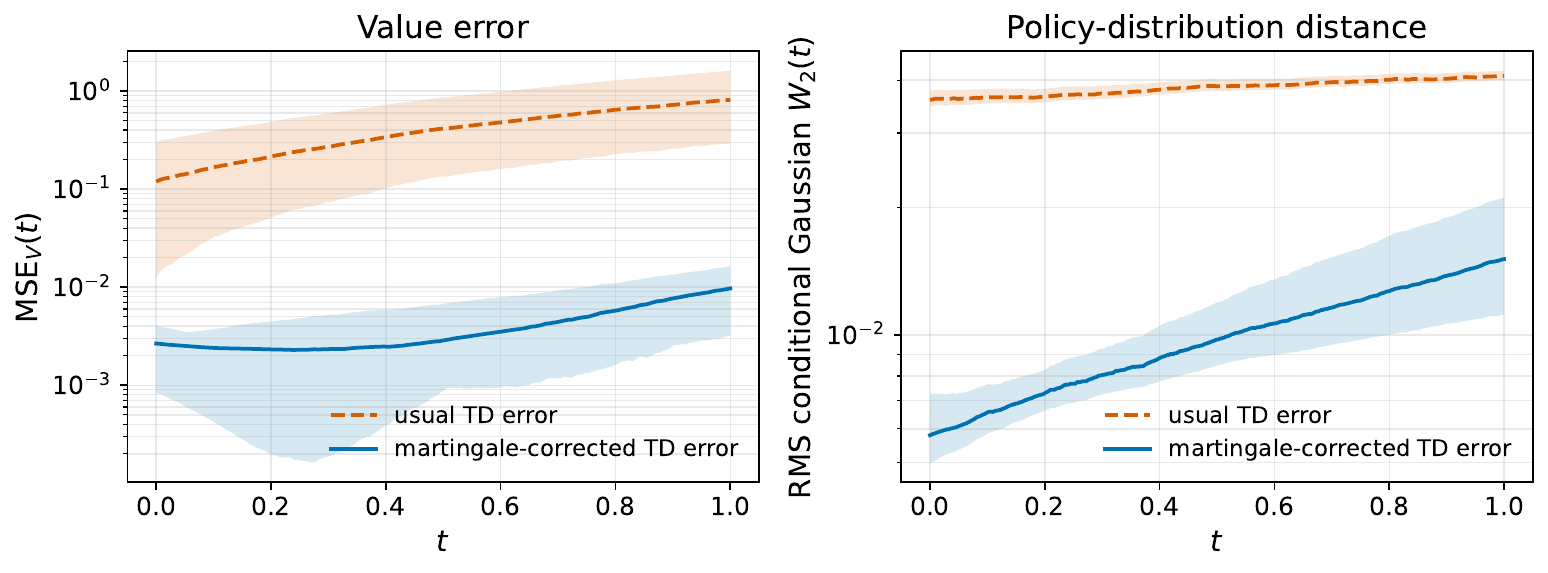}
  \caption{Comparison of the usual TD error and martingale-corrected TD error over three training seeds. Left: value-function MSE. Right: Gaussian
  policy-distribution $W_2$ distance. At each of the $501$ time points, each seed-level metric is computed over $500$ common reference trajectories; the curves show the mean across the three seeds, and the shaded bands span the minimum to maximum across these three seed-level values.}
  \label{fig:trajectory-comparison}
\end{figure}

In addition, Figure~\ref{fig:trajectory-comparison} shows the pointwise value-function and policy errors over all $501$ evaluation time points. The curves represent the across-seed mean errors, while the shaded bands are constructed from the three seed-level values at each time point, with the lower and upper boundaries given by their minimum and maximum, respectively. For these three paired seeds, the results illustrate that reducing the variance of the critic learning signal can improve the accuracy of the learned value function and policy.

\ifdefined\arxivcombined\else
\bibliographystyle{plain}
\bibliography{ref}
\fi